\documentclass{amsart}

\usepackage{amsmath,amsthm,indentfirst}
\usepackage{amssymb}
\usepackage{amsfonts}
\usepackage{amscd}
\usepackage[latin1]{inputenc}
\DeclareMathAlphabet{\mathpzc}{OT1}{pzc}{m}{it}
\theoremstyle{plain}

\newtheorem{thm}{Theorem}[section]
\newtheorem{cor}[thm]{Corollary}
\newtheorem{lem}[thm]{Lemma}
\newtheorem{prop}[thm]{Proposition}

\newtheorem{conj}[thm]{Conjecture}

\newtheorem*{maintheorem*}{Main Theorem}
\newtheorem*{thm*}{Theorem}
\newtheorem*{thma*}{Theorem A}
\newtheorem*{thmaa*}{Theorem A'}
\newtheorem*{thmb*}{Theorem B}
\newtheorem*{thmo*}{Theorem 1.1}
\newtheorem*{thmc*}{Theorem C}
\newtheorem*{thmd*}{Theorem D}
\newtheorem*{thmf*}{Theorem 4.1}
\newtheorem*{remark*}{Remark}
\newtheorem*{conjecture*}{Conjecture}
\newtheorem*{prop*}{Proposition}
\newtheorem*{lem*}{Basic Lemma}

\theoremstyle{definition}

\newtheorem*{proofc*}{Proof of Theorem C}

\newtheorem{definition}[thm]{Definition}

\newtheorem{remark}[thm]{Remark}

\def\bbz{\mathbb{Z}}

\def\bba{\mathbb{A}}
\def\bbb{\mathbb{B}}

\def\bbn{\mathbb{N}}
\def\bbg{\mathbb{G}}

\def\bbu{\mathbb{U}}
\def\bbp{\mathbb{F}}
\def\bbl{\mathbb{L}}

\def\bfr{\mathfrak{B}}

\def\vare{\varepsilon}

\def\tbf{\mathbf{t}}
\def\l{\langle}
\def\r{\rangle}

\def\la{\lambda}
\def\p{\prod_{\nu\in S}}
\def\tp{\prod_{\nu\in\tilde{S}}}
\def\u{\mathcal{U}}
\def\ubf{\mathbf{u}}

\def\v{\mathcal{V}}

\def\imp{\mbox{Im}(\phi)}
\def\impt{\mbox{Im}(\phi_2)}

\def\nhu{N_H(\u)}
\def\G{{}^{\tau}\Delta(G)}
\def\Gnu{{}^{\tau}H(\nu)}

\def\n{\mathcal{U}}
\def\nhn{N_H(\mathcal{U})}
\def\haar{m_1\times m_2}

\def\s{\mathpzc{s}}
\def\tp{\mathpzc{t}}
\def\h{\hspace{1mm}}
\def\hh{\hspace{.5mm}}

\begin{document}

\title[A joinings classification in positive characteristic] {A joining classification and a special case of Raghunathan's
conjecture in positive characteristic (with an appendix by Kevin
Wortman)}

\begin{abstract}
We prove the classification of joinings for maximal horospherical subgroups acting on homogeneous spaces without any restriction on
the characteristic. 
Using the linearization technique we deduce a special case of Raghunathan's  orbit closure conjecture.
In the appendix quasi-isometries of higher rank lattices in semisimple algebraic groups over fields of positive characteristic are 
characterized.
\end{abstract}


\author{Manfred Einsiedler and Amir Mohammadi}
\thanks {M.E. was supported by DMS grant 0622397}
\address{Mathematics Department, The Ohio State University, 231 W. 18th Avenue, Columbus, Ohio 43210.}
\address{Mathematics Dept., Yale University, 10 Hillhouse Ave. 4th floor, New Haven CT 06511,}

\maketitle

\section{Introduction}
\subsection{Statements of the main results}

Let $K$ be a global field and let $\bbg$ be a connected, simply connected, almost simple group defined over $K.$ Let $S$ be a finite set of places of $K$. We let $G=\prod_{\nu\in S}\bbg(K_{\nu})$. Furthermore we will denote $G_{\nu}=\bbg(K_{\nu})$. Recall that an arithmetic lattice compatible with the $K$-group $\bbg$ is a lattice commensurable with the subgroup of $G$ consisting of all matrices (in a particular representation as a linear group) with entries in the ring of $S$-integers.  Now fix $\Gamma_1$ and $\Gamma_2$ two irreducible arithmetic lattices of $G$ once and for all and let $H=G\times G$, and let $\bbg_1$ and $\bbg_2$ be the associated $K$-groups which we assume give rise to the same $G$. We denote by $\Delta(G)=\{(g,g):g\in G\}$ the diagonally embedded $G$ in $H$, more generally for any automorphism $\tau$ of $G$ we will use the notation ${}^{\tau}\Delta=(1\times\tau)(\Delta)$. For $i=1,2$ let $\pi_i$ denote the projection onto the $i$-th factor. Let $X_i=G/\Gamma_i$ for $i=1,2$ and let $X=X_1\times X_2=H/\Gamma_1\times\Gamma_2$. We will use $\pi_i$ to denote the projection from $X$ onto $X_i$ as well. Let $\nu\in S$ be an arbitrary place and fix a minimal parabolic subgroup $\mathbb{P}_{\nu}$ of $\bbg$ defined over $K_{\nu}$ further let $P=\mathbb{P}_{\nu}(K_{\nu}).$ Let $U_{\nu}=\mathbb{U}_{\nu}(K_{\nu
 } )$ where $\mathbb{U}_{\nu}$ is the  unipotent radical of $\mathbb{P}_{\nu}.$ 
  
  A joining $\mu$ (for the $U_\nu$-actions on $X_1$ and $X_2$) is a probability measure on $X$ that is invariant under $(u,u)$ for $u\in U_\nu$
  and satisfies that the push-forward $(\pi_i)_*(\mu)$ under the projections gives the Haar measure $m_{X_i}$ for $i=1,2$.

\begin{thm}\label{joining}
If $\mu$ is an ergodic joining for the action of $U_{\nu}$ on $X_1$ and $X_2$ then one of the following holds
\begin{itemize}
\item[(i)] $\mu$ is the Haar measure $m$ on $X$, or
\item[(ii)] $\mu$ is the $\G$-invariant measure on some closed orbit of $\G$ in $X,$ furthermore the automorphism $\tau$ is the inner automorphism induced by an element $z\in Z_{G}(U_{\nu}).$
\end{itemize} 
Furthermore, in the latter case $\mu$ is the Haar measure of the closed orbit $(g_1,g_2)F_0\Gamma_1\times\Gamma_2$, where $(g_1,g_2)\in G\times G$ is any point of the support of $\mu$, $F_0={}^{\tau_0}\Delta(G)$ consists of the $K_S$-points of a subgroup $\mathbb{F}_0$ of $\bbg_1\times \bbg_2$ defined over $K$, and finally $\tau_0$ is some inner automorphism.
\end{thm}

We remind the reader that the above generalizes Ratner's joining classification \cite{Rat1} to the positive characteristic setting and is indeed a special case of Ratner's classification theorem for measures invariant under unipotent subgroups if $G$ is a Lie group \cite{Rat5} or if $G$ is a product of algebraic groups \cite{R3, MT} of characteristic zero. However, for positive characteristic there is no general classification known.  The case of positive characteristic horospherical subgroups has been studied by the second named author \cite{AmirHoro}.  The first named author has obtained in joint work with Ghosh \cite{Einsiedler-Ghosh} a classification of measures invariant under semisimple\footnote{Quite likely
one can use the proof in \cite{MT} for any unipotent subgroup in large enough positive characteristic, but the proof in \cite{Einsiedler-Ghosh} relies on a much simpler argument which requires the assumption that the acting group is semisimple.} subgroups in sufficiently big positive characteristic. We note that in the first paper the acting group is quite large inside the ambient space and there is a rather restrictive assumption on the characteristic in the second paper.  The current work does not make any restrictions on the characteristic and is still a case where the acting group is somewhat small in comparison to the ambient space. However, our assumptions still put us in the situation where there is no need to use restriction of scalars which in the general case one will not be able to avoid. Roughly speaking the main difficulty for a general measure classification is to find a bound on the degree of the field extension used in the restriction of scalars.

Using the well-known linearization technique we obtain from the above a special case of Raghunathan's orbit closure classification, which in the case of Lie groups is known in maximal generality due to Ratner's orbit closure theorem \cite{Rat4}.

\begin{thm}\label{orbitclosure}
Let the notations and conventions be as above. Then every orbit of $\Delta(G)$ in $G\times G/\Gamma_1\times\Gamma_2$ is either closed or dense.
\end{thm}

It is easy to show that the above give also an immediate dichotomy for the product $\Gamma_1\Gamma_2$ of two lattices. Either $\Gamma_1\Gamma_2$ consists of finitely many $\Gamma_2$-cosets (i.e.\ $\Gamma_1\Gamma_2/\Gamma_2$ is a finite subset of $G/\Gamma_2$) or $\Gamma_1\Gamma_2$ is dense in $G$. In the appendix, authored by Kevin Wortman, this is used as one missing ingredient for the classification of the quasi-isometries of higher rank lattices in semisimple algebraic groups over fields of positive characteristic, e.g.\ for $\operatorname{PGL}_3(\mathbb{F}_p[t])$. We refer to the appendix for the definitions and for further comments regarding the history of this problem.

\section{Preliminary and notations}
\subsection{$K_S$-algebraic groups.} Let $K$ be a global function field of characteristic $\mathfrak{p}.$ Let $S$ be a finite set of places for any $\nu\in S$ we let $K_{\nu}$ denote the completion of $K$ with respect to ${\nu}$ and let $\varpi_{\nu}$ be a uniformizer for $K_{\nu}$ fixed once and for all. We let $K_S=\p K_{\nu}.$ We endow $K_S$ with the norm $|\cdot|=\max_{\nu\in S} |\cdot|_{\nu}$ where $|\cdot|_{\nu}$ is a norm on $K_{\nu}$ for each $\nu\in S.$ A $K_S$ algebraic group $\mathbb{A}$ (resp. variety $\mathbb{B}$) is the formal product of $\p \mathbb{A}_{\nu}$ of $K_{\nu}$ algebraic groups (resp. $\p \mathbb{B}_{\nu}$ of $K_{\nu}$ algebraic varieties). As is clear from the definition of $K_S$-varieties the usual notations from elementary algebraic geometry theory e.g.\ regular maps, rational maps, rational point etc.\ are defined componentwise, and we will take this as to be understood and use these notions without further remarks. As usual there are two topologies on $\mathbb{B}(K_S)$ the Hausdorff topology and the Zariski topology. When we refer to the Zariski topology we will make this clear. Hence if in some topological statement we do not give reference to the particular topology used, then the one which is being considered is the Hausdorff topology.  

Let  $\mathbb{A}$ be a $K_S$-algebraic group and $\mathbb{B}$ a $K_S$-algebraic subgroup. And let $A=\mathbb{A}(K_S)$ and $B=\mathbb{B}(K_S).$ For any $g\in A$ normalizing $B$ we set
$$W_B^+(g)=\{x\in B\hh|\h g^{n}xg^{-n}\rightarrow e\h\h\mbox{as}\h\h n\rightarrow-\infty\}$$
$$\begin{array}{l}W_B^-(g)=\{x\in B\hh|\h g^{n}xg^{-n}\rightarrow e\h\h\mbox{as}\h\h n\rightarrow+\infty\} \vspace{.75mm}\\ Z_B(g)=\{x\in B\hh|\h gxg^{-1}=x\}\end{array}$$
If $g\in B$ and it is clear from the context we sometimes omit the subscript $B$ above. Similarly if $g\in B$ we let $P_{g}$ (resp. $P_{g}^-$) be the set of $x\in B$ such that $\{g^{n}xg^{-n}\}_{n<0}$ (resp. $\{g^{n}xg^{-n}\}_{n>0}$) is contained in a compact subset of $G$ and define $M_{g}=P_{g}\cap P^-_g.$ Note that $W_B^{\pm}(g),\hh M_{g}$ and $Z_B(g)$ are the groups of $K_S$-points of $K_S$-algebraic subgroups $\mathbb{W}_B^{\pm}(g),\hh \mathbb{M}$ and $\mathbb{Z}_B(g)$ respectively.

Let $\bbg$ be semi-simple connected defined over $K_{S}$ and let $G=\bbg(K_S).$ Let $g\in G$ be an element such that $\mbox{Ad}\hh g$ (Ad is the adjoint representation of the algebraic group on its Lie algebra) has at least one eigenvalue of absolute value $\neq 1.$ One has $\mathbb{W}^-(g)\cdot\mathbb{M}_g\cdot\mathbb{W}^+(g)$ is Zariski open in $\bbg$ and the natural map of the product $\mathbb{W}^-(g)\times\mathbb{M}_g\times\mathbb{W}^+(g)$ to $\mathbb{W}^-(g)\cdot\mathbb{M}_g\cdot\mathbb{W}^+(g)$ is a $K_S$ isomorphism of varieties. In particular 
$$\mathcal{D}(g)=W^-(g)M_gW^+(g)=(\mathbb{W}^-(g)\cdot\mathbb{M}_g\cdot\mathbb{W}^+(g))(K_S)$$
is an open neighborhood of the identity. These are well-known facts, see for example~\cite{Pa}.
 
We say an element $e\neq s\in\mathbb{A}(K_S)$ is of class $\mathcal{A}$ if $g=(g_{\nu})_{\nu\in S}$ is diagonalizable over $K_{S}$ and for all $\nu\in S$ the component $g_{\nu}$ has eigenvalues which are integer powers of the uniformizer $\varpi_{\nu}$ of $K_{\nu}$. We will need slight generalization of this notion which we define. An element $s\in G$ is said to be from class $\mathcal{A}'$ if $s=s'\hh\gamma$ where $s'\neq e$ is an element of class $\mathcal{A}$ and $\gamma$ commutes with $s'$ and generates a compact subgroup in $G.$ Note that if $s=s'\hh\gamma$ is an element from class $\mathcal{A}'$ then it is clear from the definitions that $W^{\pm}(s)=W^{\pm}(s')$ and we also have $Z(s')=M_s.$

With these notations if $g=s\in G$ is an element from class $\mathcal{A}$ then one has $M_s=Z(s)$ and we have $W^-(s)Z(s)W^+(s)$ is a Zariski open dense subset of $G$ which contains the identity.

\subsection{Ergodic measures on algebraic varieties.}
Let $\mathbb{A}$ be a $K_S$-algebraic group acting $K_S$-rationally on a $K_S$-algebraic variety $\mathbb{M}$. Let $B$ be a subgroup of $A=\mathbb{A}(K_S)$ generated by one parameter $K_S$-split unipotent algebraic subgroups and elements from class $\mathcal{A}.$ The following is proved in~\cite{MT} (which in return relies almost directly on the behavior of algebraic orbits \cite{BZ}).

\begin{lem}\label{bz-measure}(cf.~\cite[Lemma 3.1]{MT})
Let $\mu$ be an $A$-invariant Borel probability measure on $M=\mathbb{M}(K_S)$. Then $\mu$ is concentrated on the set of $A$-fixed points in $M.$ In particular if $\mu$ is $A$-ergodic then $\mu$ is concentrated at one point.   
\end{lem}

\subsection{Homogeneous measures.}
Let $A$ be a locally compact second countable group and let $\Lambda$ be a discrete subgroup of $A$. Let $\mu$ be a Borel probability measure on $A/\Lambda$. Let $\Sigma$ be the closed subgroup of all elements of $A$ which preserve $\mu$. The measure $\mu$ is called {\it homogeneous} if there exists $x\in A/\Lambda$ such that $\Sigma\hh x$ is closed and $\mu$ is unique $\Sigma$-invariant measure on $\Sigma\hh x.$

\begin{lem}\label{normal-unimodular}(cf.~\cite[Lemma 10.1]{MT})
Let $A$ be a locally compact second countable group and $\Lambda$ a discrete subgroup of $A.$ If $B$ is a normal unimodular subgroup of $A$ and $\mu$ is a $B$-invariant $B$-ergodic measure on $A/\Lambda$ then $\mu$ is homogeneous and actually $\Sigma=\overline{B\Lambda}$.
\end{lem}


\section{Polynomial like behaviour and the basic lemma}\label{secquasi}

Let $\mu$ be a probability measure on $X$ which is invariant and ergodic under the action of some unipotent $K_{S}$-algebraic subgroup of $H.$ The idea is to use polynomial like behaviour of the action of unipotent groups on $X$ to show that if certain ``natural obstructions" do not occur then one can construct new elements which leave $\mu$ invariant as well. The idea of using polynomial like behaviour of unipotent groups to construct new invariants goes back to earlier works, e.g. Margulis' celebrated proof of Oppenheim's conjecture~\cite{Mar2} using topological arguments and Ratner's seminal work on the proof of measure rigidity conjecture~\cite{Rat2, Rat3, Rat4}. We keep the language of~\cite{MT} as it is the most suitable one in our situation.

\subsection{Construction of quasi-regular maps}
Following~\cite[Section 5]{MT} we want to construct {\it quasi-regular maps}. This is essential in our construction of extra invariance for $\mu$. We first recall the definition of a quasi-regular map. We give the definition in the case of a local field, which we will need later, the $S$-arithmetic version is a simple modification.  

\begin{definition}\label{quasiregular}(cf.~\cite[Definition 5.3]{MT})
Let $\omega$ be any place in $S.$
\begin{itemize}
\item[(i)]Let $\mathbb{E}$ be a $K_{\omega}$-algebraic group, $\mathcal{W}$ a $K_{\omega}$-algebraic subgroup of $\mathbb{E}(K_{\omega})$ and $\mathbb{M}$ a $K_{\omega}$-algebraic variety. A $K_{\omega}$-rational map $f:\mathbb{M}(K_{\omega})\rightarrow\mathbb{E}(K_{\omega})$ is called $\mathcal{W}$-{\it quasiregular} if the map from $\mathbb{M}(K_{\omega})$ to $\mathbb{V}$ given by $x\mapsto\rho(f(x))p$ is $K_{\omega}$-regular for every $K_{\omega}$-rational representation $\rho:\mathbb{E}\rightarrow\mbox{GL}(\mathbb{V})$ and every point $p\in\mathbb{V}(K_{\omega})$ such that $\rho(\mathcal{W})p=p.$
\item[(ii)] If $E=\mathbb{E}(K_{\omega})$ and $\mathcal{W}\subset E$ is a $K_{\omega}$-split unipotent subgroup then a map $\phi:\mathcal{W}\rightarrow E$ is called {\it strongly} $\mathcal{W}$-{\it quasiregular} if there exist
\begin{itemize}
\item[(a)] a sequence $g_n\in E$ such that $g_n\rightarrow e.$
\item[(b)]  a sequence $\{\alpha_n:\mathcal{W}\rightarrow\mathcal{W}\}$ of $K_{\omega}$-regular maps of bounded degree.
\item[(c)]  a sequence $\{\beta_n:\mathcal{W}\rightarrow\mathcal{W}\}$ of $K_{\omega}$-rational maps of bounded degree.
\item[(d)] a Zariski open nonempty subset $\mathcal{X}\subset\mathcal{W}$ 
\end{itemize}
such that $\phi(u)=\lim_{n\rightarrow\infty}\alpha_n(u)g_n\beta_n(u)$ and the convergence is uniform on the compact subsets of $\mathcal{X}$
\end{itemize}
\end{definition}


\noindent
We note that if $\phi$ is strongly $\mathcal{W}$-quasiregular then it indeed is $\mathcal{W}$-quasiregular.
To see this, let $\rho:E\rightarrow\mbox{GL}(W)$ be a $K_{\omega}$-rational representation and let $w\in W$ be a $\mathcal{W}$-fixed vector. For any $u\in\mathcal{X}$ we have $$\rho(\phi(u))w=\lim_{n\rightarrow\infty}\rho(\alpha_n(u)g_n)w.$$ 
Identify $\mathcal{W}$ with an affine space, as we may, thanks to the fact $\mathcal{W}$ is split. The sequence $\{\psi_n:\mathcal{W}\rightarrow W,\hspace{2mm}u\mapsto\rho(\alpha_n(u)g_n)w\}$ is a sequence of polynomial maps of bounded degree and also the family is uniformly bounded on compact sets of $\mathcal{X}$ so it converges to a polynomial map with coefficients in $K_{\omega}$. This says $\phi$ is $\mathcal{W}$-quasiregular.

\vspace{1mm}
Let us go back to the setting of Theorem~\ref{joining}, in particular we have $H=G\times G.$ We want to construct quasi-regular maps for certain unipotent groups. Let us fix some notations here, for $t\in G_{\nu}$ a diagonal element of class $\mathcal{A},$ let $\u=\Delta(W^+(t)).$ The element $s=(t,t)\in H$ is of class $\mathcal{A}$ and we have $\u\subset W^+(t)\times W^+(t)=W^+(s),$ furthermore $L=W^-(s)Z(s)(W^+(t)\times\{e\})$ is a rational cross-section for $\u$ which is invariant under conjugation by $s$.

\vspace{1mm}
We fix relatively compact neighbourhoods $\mathfrak{B}^+$ and $\mathfrak{B}^-$ of $e$ in $W^+(s)$ and $W^-(s)$ respectively with the property that $\mathfrak{B}^+\subset s\mathfrak{B}^+s^{-1}$ and $\mathfrak{B}^-\subset s^{-1}\mathfrak{B}^-s.$ We define a filtration in $W^+(s)$ and $W^-(s)$ this is done by setting 
$\mathfrak{B}_n^+= s^n\mathfrak{B}^+s^{-n}$ and $\mathfrak{B}_n^-= s^{-n}\mathfrak{B}^-s^{n}$ 
respectively. For any integer $n$ we set $\u_n=\mathfrak{B}_n^+\cap\u.$ Define $\ell^{\pm}:W^{\pm}(s)\rightarrow\bbz\cup\{-\infty\}$, by
\begin{itemize}
\item[(i)] $\ell^+(x)=k$ iff $x\in\mathfrak{B}_k^+\setminus\mathfrak{B}_{k-1}^+$ and $\ell^+(e)=-\infty,$
\item[(ii)] $\ell^-(x)=k$ iff $x\in\mathfrak{B}_k^-\setminus\mathfrak{B}_{k-1}^-$ and $\ell^-(e)=-\infty.$
\end{itemize}

\noindent
As the definition suggests these functions measure the ``size" of elements in $W^{\pm}(s)$ with respect to the action of $s.$
 
Let $\{g_n\}$ be a sequence in $L\hspace{.5mm}\u\setminus N_H(\u)$ with $g_n\rightarrow e.$ Now as $L$ is a rational cross-section for $\u$ in $H$ we get rational morphisms $\tilde{\phi}_n:\u\rightarrow L$ and $\omega_n:\u\rightarrow\u$ such that $ug_n=\tilde{\phi}_n(u)\omega_n(u)$ holds for all $u$ in a Zariski open dense subset of $\u.$ 

\noindent
We now want to linearize the $\u$-action, this is done with the aid of a theorem of Chevalley. Let $\rho:H\rightarrow\mbox{GL}(V)$ be a $K_S$-representation such that \begin{equation}\label{e;chevalley1}\u=\{x\in H|\h\rho(x)v=v\}\end{equation}
for some $v\in V.$ According to this description we have 
\begin{equation}\label{e;chevalley2}\rho(N_H(\u))v=\{x\in Hv|\h\rho(\u) x=x\}.\end{equation}

\vspace{1mm}
\noindent
Let $\mathcal{B}(v)\subset V$ be a bounded neighborhood of $v$ such that 
$$\rho(H)v\cap\mathcal{B}(v)=\overline{\rho(H)v}\cap\mathcal{B}(v).$$
As $g_n\notin N_H(\u)$, there is a sequence of integers $\{b(n)\}$ with $b(n)\rightarrow\infty$ and $\rho(\u_{b(n)+1}g_n)v\not\subset\mathcal{B}(v)$ and $\rho(\u_{k}g_n)v\subset\mathcal{B}(v)$ for all $k\leq b(n).$ Define $K_{\nu}$-regular isomorphisms $\alpha_n:\u\rightarrow\u$ as follows: for every $u\in\u$ 
$$\lambda_n(u)=
s^{n}us^{-n}\hspace{3mm}\mbox{and set}\hspace{3mm}\alpha_n=\la_{b(n)}$$

The $K_{\nu}$-rational maps $\phi_n$'s are then defined by $\phi_n=\tilde{\phi}_n\circ\alpha_n:\u\rightarrow L.$ Using these we define
$$\phi'_n=\rho_L\circ\phi_n:\u\rightarrow V$$
where $\rho_L$ is the restriction to $L$ of the orbit map $h\mapsto \rho(h)v.$ Notice that by construction of $b(n)$ we have $\phi'_n(\mathfrak{B}_0)\subset\mathcal{B}(v)$ but $\phi'_n(\mathfrak{B}_1)\not\subset\mathcal{B}(v)$. 

As $\phi'_n(u)=\rho(\alpha_n(u)g_n)v$ we see that $\phi'_n:\u\rightarrow V$ is a $K_{\nu}$-regular map. If we use the fact that $\u$ is split we may identify $\u$ with the affine space of the same dimension via a polynomial map. This says we can interpret $\{\phi'_n\}$ as a set of $K_{\nu}$-polynomial maps of bounded degree. Using the definition of $\phi'_n$ we have $\{\phi'_n\}$ is uniformly bounded family of polynomials of bounded degree, thus passing to a subsequence, which we will still write as $\phi'_n,$ we may assume there is a $K_{\nu}$-regular map $\phi':\u\rightarrow V$ such that 
\begin{equation}\label{conv-eq}
 \phi'(u)=\lim_{n\rightarrow\infty}\phi'_n(u)\hspace{3mm}\mbox{for every}\hspace{3mm}u\in\u.
 \end{equation} 
Note that $\phi'(e)=v$ as $g_n\rightarrow e$ and that $\phi'$ is non-constant since $\phi'(\overline{\mathfrak{B}_1})\not\subset\mathcal{B}(v)^\circ$.

As $L$ is a rational cross-section for $H/\u$ we have that $L$ gets mapped onto a Zariski open dense subset $\mathcal{M}$ of the Zarsiki closure of $\rho(H)v$ and that $v\in\mathcal{M}.$ So we can define a $K_{\nu}$-rational map $\phi:\u\rightarrow L$ by 
$$\phi=\rho_L^{-1}\circ\phi'$$
The construction above gives $\phi(e)=e$ and that $\phi$ is non-constant.

We now show that the map $\phi$ constructed above is strongly $\u$-quasiregular.
Note that by the above construction we have for $u\in\phi'^{-1}(\mathcal{M})$ 
$$\phi(u)=\lim_{n\rightarrow\infty}\phi_n(u)$$
and the convergence above is uniform on the compact subsets of $\phi'^{-1}(\mathcal{M})$ (as \eqref{conv-eq} is uniform
on compact subsets and $\rho_L^{-1}$ is continuous on compact subsets of $\mathcal{M}$). We have 
$$\phi_n(u)=\alpha_n(u)g_n\beta_n(u)\hspace{3mm}\mbox{where}\hspace{3mm}\beta_n(u)=\omega_n(\alpha_n(u))^{-1}$$
The above says for $u\in\phi'^{-1}(\mathcal{M})$ we can write
\begin{equation}\label{e;phi-conv} \phi(u)=\lim_{n\rightarrow\infty}\alpha_n(u)g_n\beta_n(u),\end{equation} 
as we wished to show.

\subsection{Properties of $\phi$}
We now recall some important properties of the map $\phi$ constructed above. The proofs are more or less similar to the proofs in~\cite[section 6]{MT}. We will remark on how one translates the proofs from characteristic zero setting in~\cite{MT} to our setting here. In most cases the proofs simplify a bit in our setting thanks to the fact that our $\u$ is quite special.

Let us start with the following important property of the map $\phi.$ These will be used in various part in the proof of Theorem~\ref{joining}.

\begin{prop}\label{normalizer}(cf. ~\cite[Proposition 6.1]{MT})

$\phi$ maps $\u$ into $\nhu$, furthermore there is no compact subset $C$ of $H$ such that $\mbox{\rm{Im}}(\phi)\subset C\u.$
\end{prop}

\begin{proof}
We will use the notation as above. Recall from~\eqref{e;chevalley2} that $N_H(\u)=\{h\in H:\h \rho(\u)\rho(h)v=\rho(h)v\}$. Thus we need to show that for any $u_0\in \u$\vspace{1mm} we have $\rho(u_0)\rho(\phi(u))v=\rho(\phi(u))v$ for all $u\in\u$. 

Let $u\in\phi'^{-1}(\mathcal{M}).$ We saw in~\eqref{e;phi-conv} that $\phi(u)=\lim_{n\rightarrow\infty}\alpha_n(u)g_n\beta_n(u).$ We have
$$\rho(u_0\alpha_n(u)g_n)v=\rho(\alpha_n(\alpha_{n}^{-1}(u_0)u)g_n)v$$
Note that $\alpha_n^{-1}(u_0)\rightarrow e$ as $n\rightarrow\infty$ thus we have $\phi(u)\in N_H(\u)$ for all $u\in\phi'^{-1}(\mathcal{M}).$ The result now follows since $\phi'^{-1}(\mathcal{M})$ is a Zariski dense subset of $\u.$ 

To see the second assertion note that $\phi=\rho_L^{-1}\circ\phi'$ and $\phi'$ is a non-constant (hence unbounded) polynomial map.
\end{proof}

The following is a technical condition on the sequence $\{g_n\}$ which is needed in the proof of Basic Lemma below, it is Definition 6.6 in~\cite{MT}.
 
\begin{definition}\label{cstar}
A sequence $\{g_n\}$ is said to satisfy the condition $(*)$ with respect to $s$ if there exists a compact subset $C$ of $H$ such that for all $n\in\bbn$ we have $s^{-b(n)}g_ns^{b(n)}\in C.$ 
\end{definition}

Let us fix some further notations, denote $\mathcal{V}=W^+(t)\times \{e\}\subset W^+(s)$ and the corresponding counter parts $\u^-=\Delta(W^-(t))$ and $\v^-=W^-(t)\times\{e\},$ which are subgroups of $W^-(s).$ Note that $$\mathcal{D}=\u^-\v^-Z(s)\v\u=W^-(s)Z(s)W^+(s)$$ is a Zariski open dense subset of $H$ and for any $g\in\mathcal{D}$ we have a unique decomposition 
\begin{equation}\label{decomposition} g=w^-(g)z(g)w^+(g)=u^-(g)v^-(g)z(g)v(g)u(g)\end{equation}
where $u^-(g)\in\u^-,\h v^-(g)\in\v^-,\h z(g)\in Z(g),\h u(g)\in\u,\h v(g)\in\v,\h w^-(g)=u^-(g)v^-(g)$ and $w^+(g)=v(g)u(g).$

Note that for every $w^{\pm}\in W^{\pm}(s)$ we have 
\begin{equation}\label{e;ell-func}\ell^{\pm}(s^kw^{\pm}s^{-k})=\ell^{\pm}(w^{\pm}(g))\pm k\end{equation}  

Recall that $\u$ is normalized by $s$ by our definitions. 


\begin{prop}\label{star} (cf. ~\cite[Proposition 6.7]{MT})
Let $s$ and $\u$ be as above then the following hold
\begin{itemize}
\item[(i)] any sequence $\{g_n\}$ satisfies condition $(*)$ with respect to $s.$
\item[(ii)] if the sequence $\ell^-(v^-(g_n))-\ell^-(u^-(g_n))$ is bounded from below, then $\phi(\u)\subset W^+(s).$
\end{itemize}
\end{prop}

\begin{proof}
Denote $\ell^-(w^-(g_n))=k(n).$ Let $h_n=s^{k(n)}g_ns^{-k(n)}.$ It follows from~\eqref{e;ell-func} that $h_n\in \mathfrak{B}^-_0\setminus\mathfrak{B}^-_{-1}.$ Since $g_n\rightarrow e$ we have $\lim_n w^+(h_n)=\lim_n z(h_n)=e.$ Hence passing to a subsequence we may and will assume $\lim_nh_n=h\in W^-(s)$ and $h\neq e.$ Now the same argument as in the proof of~\cite[Proposition 6.7]{MT} with $U=U_0=\u$ gives: if $\{g_n\}$ does not have property $(*)$ then 
\begin{equation}\label{e;uhw^-u}\u\hh h\subset\overline{W^-(s)\u}\end{equation} 
where the closure is the Zariski closure. Let $h=(h_1,h_2).$ Since $h\neq e$ either $h_1\neq e$ or $h_2\neq e,$ assume $h_1\neq e.$ If we now project~\eqref{e;uhw^-u} to the first component and use the fact that $\u=\Delta(W_G^+(t)),$ we have $W_G^+(t)h_1\subset\overline{W_G^-(t)W_G^+(t)}.$ However since $h_1\in W_G^-(t)$ it follows from general facts about semisimple groups that $W_G^+(t)h_1$ must have components in $Z_G(t)$ which contradicts~\eqref{e;uhw^-u}. Let us give a proof of this using the proof of part (b) of~\cite[Proposition 6.7]{MT} without referring to these algebraic group facts. Let $\mathbb{E}=\{x\in H:\h \u x\subset \overline{W^-(t)\u}\}.$ Note that $\mathbb{E}=\{x\in H:\h \overline{W^-(t)\u} x\subset \overline{W^-(t)\u}\}$ hence this is a $K_S$-closed algebraic group, see~\cite{B1}. It is shown in loc. cit. that $\mathbb{E}$ is a unipotent algebraic subgroup. Let $E=\mathbb{E}(K_S)$ then the Zariski closure of $E$ which we denote by $\mathbb{E}'$ is a unipotent algebraic group defined over $K_S$ and $E=\mathbb{E}'(K_S),$ see~\cite[AG 12, 14]{B1}. Now since $h\in W^-(s)\setminus\{e\}$ and $h\in E$ we have $E\neq\u.$ Since $E$ is unipotent this implies that $N_E(\u)\neq\u.$ Note that $E\subset\overline{W^-(s)\u}$ and that $W^-(s)\u$ is Zarsiki open subset of its closure containing the identity. Thus $N_H(\u)\cap W^-(s)\neq\{e\}.$ This contradicts the fact that $N_H(\u)\subset Z_H(s)W^+(s)$ and finishes the proof of (i). 

The proof of part (ii) is the same as the proof given in~\cite[Proposition 6.7]{MT}, see page 368 in loc. cit.   
\end{proof}

As before we have $s=(t,t)$ and $\u=\Delta(W^+(t))$. Now let $\mu$ be a $\u$-invariant probability measure on $X$ and let $\mu=\int_{Y}\mu_{y}d\sigma(y)$ be an ergodic decomposition for $\mu.$ For $\mu$-a.e. $x\in X$ we set $y(x)$ to be the corresponding point from $(Y,\sigma).$ We also fix once and for all, the left invariant Haar measure  $\theta$ on $\u$ note that as $\u$ is unipotent $\theta$ is also the right invariant Haar measure on $\u.$

\begin{definition}\label{aver}
A sequence of measurable non-null sets $A_n\subset \u$ is called an {\it averaging net} for the action of $\u$ on $(X,\mu)$ if the following analogue of the Birkhoff pointwise ergodic theorem holds. For any continuous compactly supported function $f$ on $X$ and for almost all $x\in X$ one has
$$\lim_{n\rightarrow\infty}\frac{1}{\mu(A_n)}\int_{A_n}f(ux)d\theta(u)=\int_Xf(h)d\mu_{y(x)}(h).$$ 
\end{definition}

\begin{lem}\label{averlem}(cf.\ \cite[section 7.2]{MT})
Let $A\subset\u$ be relatively compact and non-null. Let $A_n=\la_n(A).$ Then  $\{A_n\}$ is an averaging net for the $\u$ action on $(X,\mu).$
\end{lem}

We note that if we choose $A$ to be a compact subgroup with $A\subset sAs^{-1},$ then this lemma follows from the decreasing Martingale theorem.
  
\begin{definition}\label{unifconv}
$\Omega\subset X$ is said to be {\it a set of uniform convergence relative to $\{A_n\}$} if for every $\vare>0$ and every continuous compactly supported function $f$ on $X$ one can find a positive number $N(\vare,f)$ such that for every $x\in\Omega$ and $n>N(\vare,f)$ one has
$$\left|\frac{1}{\mu(A_n)}\int_{A_n}f(ux)d\theta(u)-\int_Xf(h)d\mu_{y(x)}(h)\right|<\vare.$$ 
\end{definition}

It is an easy consequence of Egoroff's theorem and second countablity of the spaces under consideration and is proved in~\cite[section 7.3]{MT} that for any $\vare>0$ one can find a measurable set $\Omega$ with $\mu(\Omega)>1-\vare$ which is  a set of uniform convergence relative to $\{A_n=\la_n(A)\}$ for every relatively compact non-null subset $A$ of $\u.$

The following is the main point of the construction of the quasi-regular maps and provides us with the extra invariance that we were after in this section.

\begin{lem*}\label{basic}(cf.\ \cite[Basic Lemma 7.5]{MT})\\
Let $\Omega$ be a set of uniform convergence relative to averaging nets $\{A_n=\la_n(A)\}$ corresponding to arbitrary relatively compact non-null subset $A\subset\u.$ Let $\{x_n\}$ be a sequence in $\Omega$ with $x_n\rightarrow x\in\Omega.$ Let $\{g_n\}\subset H\setminus\nhu$ be a sequence which satisfies condition $(*)$ with respect to $s.$ Assume further that $g_nx_n\in\Omega$ for every $n.$ Now if $\phi$ is the $\u$-quasiregular map corresponding to $\{g_n\}$ constructed above then the ergodic component $\mu_{y(x)}$ is invariant under $\mbox{\rm{Im}}(\phi).$   
\end{lem*}

\begin{proof}
The same proof as in~\cite[Basic Lemma]{MT} with $U=U_0=\u$ and $p=id$ works here.  
\end{proof}

\begin{remark}\label{agoestoa'}
In this section we carried out the construction for elements from class $\mathcal{A},$ however one sees from the construction that this could be done following exact the same lines when $s$ is an element from class $\mathcal{A}'.$ We will use this without further notice.  
\end{remark}


\section{Joining classification for the action of $U_{\nu}$}\label{secjoining}
Recall that a joining for the action of $U=U_{\nu}$ on $X_i$ is an $\n$-invariant probability measure on $X$ which satisfies the property 
$$\pi_{i*}(\mu)=m_i\hspace{3mm}\mbox{for}\hspace{3mm}i=1,2$$  

In this section we use the construction and the properties from the previous section and complete the proof of Theorem~\ref{joining}. Let the notations and conventions be as before in particular $U=U_{\nu}$ is a maximal unipotent subgroup of $G_{\nu}.$ We let $\n=\Delta(U).$ As in the statement of Theorem~\ref{joining} let $\mu$ be an ergodic joining for the action of $U$ on $X_i.$

\begin{remark}\label{ergodic-joining}
Observe that if $\mu$ is a joining for the action of $U$ on $X_i$'s then almost every ergodic component is a joining as well. To see this let $\mu=\int_{Y}\mu_y\h d\sigma(y)$ be an ergodic decomposition of $\mu.$ Now as $\pi_{i*}\mu=\int_Y\pi_{i*}\mu_yd\sigma(y)$ and since $\pi_i(\n)$ acts ergodically on $X_i$ we get for $\sigma$-a.e $\mu_y$ is a joining for $\n.$ 
\end{remark}

Note that our subgroup $\n$ here fits into the general framework of Section~\ref{secquasi}. That is there is a diagonal element $t\in G_{\nu}$ such that $U=W^+(t)$ and $\n=\Delta(U).$ Set ${s}=(t,t)$. 
We let $L=W^-(s)Z(s)(U\times\{e\})$ be a rational cross-section for $\n$ in $H,$ as before. Note that indeed $P_{\nu}=N_{G_{\nu}}(U)$ is the minimal parabolic subgroup in the introduction. Furthermore we remark that in this setting one has the following description 
\begin{equation}\label{nhn-eq}
 \nhn=\Delta(P_{\nu})(Z_{G}(U)\times Z_G(U)).
\end{equation}  
To see this  note that $N_G(U)=P$ and so an element $(g_1,g_2)\in \nhn$ must satisfy $g_1,g_2\in P$.
In fact, from $(g_1,g_1)^{-1}(g_1,g_2)=(e,g_1^{-1}g_2)\in\nhn$ we get that $g_1^{-1}g_2\in Z_G(U)$
which gives the claim.

To state the next lemma we need to fix some further notations. Let $\mathbb{T}$ be a maximal torus of $\mathbb{G}$ defined over $K_{\nu},$ which normalizes $\mathbb{U}.$ Let $\mathbb{B}$ be a Borel subgroup of $\mathbb{G}$ (defined over $\overline{K}$ the algebraic closure of $K$) containing $\mathbb{T}$ and so that $\mathbb{U}$ consists of positive roots. Throughout we fix an ordering on the roots induced by $\mathbb{B}.$ We may and will assume that this ordering is compatible with the natural ordering induced by action of $t$ which was used in the previous section. Let $\mathbb{S}$ be maximal $K_{\nu}$-split subtorus of $\mathbb{T}$ and let $S=\mathbb{S}(K_{\nu}).$ Let $\Phi$ be the root system corresponding to $\mathbb{T}$ above and let $_{\nu}\Phi$ be the relative root system corresponding to $\mathbb{S}.$ 
Let $\gamma\in {}_{\nu}\Phi$ be a dominant root with respect to $_{\nu}\Phi.$ Let $\bbu_{\gamma}$ be the unique connected unipotent subgroup defined and split over $K_{\nu},$ normalized by $Z_{\mathbb{G}}(\mathbb{S})$ corresponding to the relative root $\gamma.$ This is  $Z(\bbu_{(\gamma)})$ if $\gamma$ is a multiple relative root, in the notation of~\cite{B1}. Note that $\bbu_{\gamma}$ is $K_{\nu}$-isomorphic to an affine space. And we let $\bbu_{-\gamma}$ be the corresponding object with respect to $-\gamma.$ We chose $\gamma$ to be dominant root in $_{\nu}\Phi.$ Thus we may invoke results from~\cite[Sect.~3]{BT1} on properties of groups generated by roots corresponding to quasi-closed (terminology as in loc.\ cit.) subsets of $_{\nu}\Phi$. We obtain that the algebraic group $\bbg_{\gamma}=\l \bbu_{\gamma}, \bbu_{-\gamma}\r$ generated by $\bbu_{\gamma}$ and $\bbu_{-\gamma}$ is defined over $K_{\nu}$ and has $K_{\nu}$-rank equal one. We let $U_{\gamma}=\bbu_{\gamma}(K_{\nu}
 ),$ $U_{-\gamma}=\bbu_{-\gamma}(K_{\nu})$ and $G_{\gamma}=\bbg_{\gamma}(K_{\nu})$


\vspace{1mm}
Now fix a cross-section, $U_{-\gamma}'$, for $U_{-\gamma}$ in $W^-(t)$ defined over $K_{\nu}$ and invariant under conjugation by $S.$ Such cross-section exists, see for example in~\cite{BS}. If $g\in W^-(t)$ we write $g=g_{\gamma}g'$ where $g_{\gamma}\in U_{-\gamma}$ and $g'\in U_{-\gamma}'$ is in the fixed cross-section.

We equip $G_{\nu}$ with a right invariant metric, $d_r(\h,\h)$ (which near $e$ we define via the matrix norm $\operatorname{Mat}_\ell(K_\nu)\supset G_\nu$ by averaging over a ``good" compact open subgroup of $G_\nu$).  Denote $|g|=d_r(e,g).$ Now as $\gamma$ is the highest root and since $t$ was chosen to be regular i.e.\ $\alpha(t)\neq1$ for all $\alpha\in{} _{\nu}\Phi,$ we may find $a>1$ depending on $_{\nu}\Phi$ and $t$ with the following property
for any $\kappa_n\rightarrow0$: If $\{h_n\}$ is a bounded sequence in $G_{\nu}$ which satisfies 
$$\kappa_n^a<|w^-(h_n)_{\gamma}|<\kappa_n\h\h\mbox{and}\h\h|w^-(h_n)|<\kappa_n$$ 
then $\ell^-(w^-(h_{n})_{\gamma})-\ell^-(w^-(h_{n})')$  tends to $+\infty$ as $n$ tends to infinity. Recall that the function $\ell^-$ measures the expansion factor of the action of $s^{-1}$ on $W^-(s).$ So this assertion is the fact that $-\gamma$ is expanded the most! Let us fix $\kappa_n=\frac{1}n$.

\vspace{1mm}
\begin{lem}\label{sequ}
There exists some $0<\vare<1$ with the following property for any $\Omega$ which satisfies $\mu(\Omega)>1-\vare.$ There exists a sequence $\{g_n\}$ such that 
\begin{itemize}
\item[(i)] $d(g_n,e)<\kappa_n\rightarrow 0$ 
\item[(ii)] $g_n\Omega\cap\Omega\neq\emptyset$ 
\item[(iii)] If $g_n=(g_{1,n},g_{2,n})$ and $g_{2,n}=w^-(g_{2,n})z(g_{2,n})w^+(g_{2,n})$ and we write $w^-(g_{2,n})=w^-(g_{2,n})_{\gamma}w^-(g_{2,n})'$ as above, then for all large enough $n$ 
$$|g_{1,n}|<\kappa_n^{\h\frac{a-1}{2\dim\bbg}},\h\h\kappa_n^a<|w^-(g_{2,n})_{\gamma}|<\kappa_n\h\h\mbox{and}\h\h |w^-(g_{2,n})'|<\kappa_n$$
\end{itemize}
\end{lem}

This lemma is one of the places where we make use of the assumption that $\mu$ is a joining.

\begin{proof}
For each $n\in\bbn$ we need to find $g_n$ which \vspace{.5mm}satisfies (i), (ii), (iii) above. 

For now we fix $n$ and to further simplify notations we write  $\kappa=\kappa_n$ and $\eta=\kappa_n^{\h\frac{a-1}{2\dim\bbg}}.$ We induce a metric on $X$ with the aid of a right invariant metric on $G_{\omega}$ for all $\omega\in S.$ Fix $K_i$ two relatively compact open subsets of $X_i$ respectively such that $\Omega\subset K_1\times K_2.$ Recall that $\mu$ is a joining thus we have $\pi_{2*}(\mu)=m_2.$ If $\Omega_2=\pi_2(\Omega)$ then $\Omega_2$ is a relatively compact set and $m_2(\Omega_2)>1-\vare.$ Taking $n$ large enough we may assume that the ball of radius $\eta$ around each point in $K_i$ is the injective image of the corresponding ball in $G.$ Let $\Omega_2=\bigcup_i B_i$ be a disjoint union of balls of radius $\kappa.$ Further let
\begin{equation}\label{equ:R}\pi_2^{-1}(B_i)\cap \bigl(K_1\times K_2\bigr)=\bigcup_{j\in J(i)} R_i^j\times B_i\end{equation} 
be a disjoint union, where each $R_i^j$ is contained in a ball of radius $2\eta$ in $K_1.$ Note that the number $\#J(i)$ of such sets $R_i^j$ that are needed, is bounded by $c_1\hh\eta^{-\dim\mathbb{G}},$ where $c_1$ is a constant depending on $\Omega.$ 

Using~(\ref{equ:R}) and the fact that in a non-archimedean metric space any point of a ball is the center we see that for all $x_i^j\in R_i^j\times B_i$ we have $B_i=B_{\kappa}^G(\pi_2(x_i^j)).$

Define the set
$$\mathcal{N}_{\kappa}=\{g\in B_{\kappa}^G(e)\h|\h g=w^-(g)_{\gamma}w^-(g)'z(g)w^+(g)\h\mbox{and}\h w^-(g)_{\gamma}\in B_{\kappa^a}^G(e)\}.$$
Now assume the opposite to the lemma that is 
$$\mbox{for all}\h\hh x_i^j\in(R_i^j\times B_i)\cap\Omega\h\hh\mbox{one has}\h\hh (R_i^j\times B_i)\cap\Omega\subseteq\pi_2^{-1}(\mathcal{N}_{\kappa}\pi_2(x_i^j)\cap B_i).$$
This in turn will give
$$\pi_2^{-1}(B_i)\cap\Omega=\bigcup_j\hh (R_i^j\times B_i)\cap\Omega\subseteq\pi_2^{-1}\bigl(\hh\bigcup_j\mathcal{N}_{\kappa}\pi_2(x_i^j)\cap B_i\hh\bigr).$$
It follows from the definition of $\mathcal{N}_{\kappa}$ that $m_2(\mathcal{N}_{\kappa})\leq c_2\kappa^a\kappa^{\dim\bbg-1}=c_2\kappa^{a-1}m_2(B_i)$
for some constant $c_2$ that only depends on the Haar measure $m_2$.  Hence we get 
$$m_2\bigl(\hh\bigcup_j\mathcal{N}_{\kappa}\pi_2(x_i^j)\cap B_i\hh\bigr)\leq (\#\hh J)\hh c_2\kappa^{a-1}m_2(B_i)$$
We now have 
$$\Omega=\bigcup_i\bigcup_{j}(R_i^j\times B_i)\cap\Omega\subseteq\bigcup_i\pi_2^{-1}(\hh\bigcup_j\mathcal{N}_{\kappa}Y\pi_2(x_i^j)\cap B_i\hh).$$
As the balls $B_i$ were chosen to be disjoint this gives
$$1-\vare<\mu(\Omega)\leq\sum_i\mu(\pi_2^{-1}(\hh\bigcup_j\mathcal{N}_{\kappa}\pi_2(x_i^j)\cap B_i\hh))\leq c\hh\eta^{-\dim\mathbb{G}}\kappa^{a-1}$$
Our choice of $\eta$ and $\kappa$ now says for small enough $\eta$ and $\kappa$, i.e.\ for large enough $n,$ the right hand side of the above inequality is less than $1-\vare.$ This contradiction finishes the proof.
\end{proof}
 
The next proposition provides us with the main ingredient to apply entropy arguments, i.e.\ it will provide us an element of class $\mathcal{A}'$ which leaves the measure invariant and does not contract $\n$. In order to prove this proposition we apply the construction recalled in Section~\ref{secquasi} corresponding to the above constructed displacements.

\begin{prop}\label{torusinv}
Let $\n$ and $\mu$ be as above then $\mu=m_1\times m_2$ (in which case the statements below hold trivially also) or
\begin{itemize}
\item[(i)] there exists $\s=({}^{z}\tp,\tp)\in H$ an element of class $\mathcal{A}',$ where $z\in Z_{G_{\nu}}(U)$ and $\tp\in Z_{G_{\nu}}(t)$ such that $\mu$ is $\s$-invariant, and
\item[(ii)] $\n\subset Z_{H}(\s)W^+(\s)$.
\end{itemize}
\end{prop}

\begin{proof}
Let $\{g_n\}$ be a sequence which satisfies the claims in Lemma~\ref{sequ}. We construct the quasi-regular map $\phi:\n\rightarrow L$ as in 
Section~\ref{secquasi} corresponding to this sequence $\{g_n\}.$ Property (iii) in Lemma~\ref{sequ} gives $\{g_n\}\subset H\setminus\nhn.$ An application of the Basic Lemma says that $\mu$ is invariant under $\mbox{Im}(\phi).$ Recall that by Proposition~\ref{normalizer} we are also guaranteed that $\mbox{Im}(\phi)\subset\nhn.$ Let us denote $\phi=(\phi_1,\phi_2)$ as $\phi$ is non-constant at least one of $\phi_i$'s are non-constant. There are two cases to consider

\textbf{\textit {Case 1:}} There exists $i$ such that $\phi_i$ is constant. In this case we claim that $\mu=m_1\times m_2$ is the Haar measure on $X.$

 \textit {Proof of the claim:} We give the proof in the case $\phi_2$ is constant, the proof in the other case is identical. As $\phi(e)=e$ we get $\phi_2(u)=e$ for all $u\in\n$ so we have $\phi(u)=(\phi_1(u),e)$ and as $\phi$ is unbounded quasi-regular we have $\phi_1$ is unbounded quasi-regular.

Recall that $X= G/\Gamma_1\times G/\Gamma_2= X_1\times X_2.$ Define the $\sigma$-algebra
$$\Xi=\{X_1\times \xi\h|\h\xi\h\mbox{is a Borel set in}\h X_2\}.$$ 
We let $\mu_x^{\Xi}$ denote the conditional measures, these are probability measures on $[x]_{\Xi}=X_1\times\{\pi_2(x)\}$ and one has
\begin{equation}\label{dis-int}
 \mu=\int_{X}\mu_x^{\Xi}\h d\mu
\end{equation} 
Since $\phi=(\phi_1,e)$ every $A\in\Xi$ is invariant under $\imp.$  This implies together with $\phi(u)$ preserving $\mu$ that all the conditional measures $\mu_x^{\Xi}$ are invariant under $\imp.$ Applying the push forward map $\pi_{1*}$ to \eqref{dis-int} we get
$$\pi_{1*}(\mu)=\int_{X}\pi_{1*}(\mu_x^{\Xi})\h d\mu.$$ 
But $\mu$ is a joining for the action of $U,$ so we have $\pi_{1*}(\mu)=m_1$ which using the above says that
$$m_1=\int_{X}\pi_{1*}(\mu_x^{\Xi})\h d\mu.$$
We now recall that 
\begin{itemize}
\item[(i)] $\phi=(\phi_1,e)$ is unbounded quasi-regular map.
\item[(ii)] $\pi_{1*}(\mu_x^{\Xi})$ is $\imp$-invariant for all $x\in X$
 \item[(iii)] $\Gamma_1$ is an irreducible lattice in $G$ and $G$ is simply connected so $m_1$ is ergodic under any unbounded subgroup of $G.$  
\end{itemize}
putting all these together says $\mu$-a.e. $\mu_x^{\Xi}=m_1\times\delta_{\pi_2(x)}$ so we have $\mu=m_1\times\pi_{2*}(\mu)$ appealing to $\pi_{2*}(\mu)=m_2$  gives that $\mu=m_1\times m_2$ and the claim is proved.

\textbf{\textit {Case 2:}} $\phi_1$ and $\phi_2$ are both non-constant. In this case we will use the particular structure of the sequence $g_n$ in Lemma~\ref{sequ} to prove the existence of $\s$ as in the statement of the proposition.

Recall that $\phi=(\phi_1,\phi_2):\n\rightarrow L\cap\nhn,$ where $$L=W^-(s)Z_H(s)(U\times\{e\})$$ is the cross-section we chose in Section~\ref{secquasi}.
Also recall from \eqref{nhn-eq} that $\nhn=\Delta(P)(Z_G(U)\times Z_G(U))$ so that together we have
\begin{align*}L\cap\nhn =& [W^-(s)Z_H(s)(U\times\{e\})]\cap[\Delta(P)(Z_G(U)\times Z_G(U))] \\ =& [Z_H(s)(U\times\{e\})]\cap[\Delta(P)(Z_G(U)\times Z_G(U))].
\end{align*}  
This says $\phi_2(u)\in Z_G(t)\cap G_{\nu}$ and that $\phi_1(u)=\phi_2(u)\phi^Z_1(u)$ on a Zariski open dense subset of $\n,$ where $\phi_1^Z(u)\in Z_{G_{\nu}}(U).$ 

We show that there exists $u_0$ such that the element $\tp=\phi_2(u_0)$ is from class $\mathcal{A}'$ and  satisfies $U\subset Z_G(\tp)W^+(\tp).$ 


We first take a more careful look at the construction of $\phi.$ Recall that we have
\begin{equation}\label{quasiphi}\phi(u)=\lim_{n\rightarrow\infty}\phi_{n}(u)=\lim_{n\rightarrow\infty}\alpha_n(u)g_n\beta_n(u),\end{equation} 
where $\alpha_n: \n\rightarrow \n$ is a regular map and $\beta_n: \n\rightarrow\n$ is a rational map. 
More precisely $\alpha_n(u)=s^{b(n)}us^{-b(n)}$ where the sequence $b(n)$ consists of renormalization constants that are defined in Section~\ref{secquasi}. 
We now concentrate our attention at $\phi_2$. Taking the second component of \eqref{quasiphi} we get
\begin{equation}\label{quasiphi2'}
 \phi_2(u)=\lim_{n\rightarrow\infty}\phi_{2,n}(u)=\lim_{n\rightarrow\infty}t^{b(n)}u\cdot(t^{-b(n)}g_{2,n}t^{b(n)})\cdot t^{-b(n)}\beta_{2,n}(u),
\end{equation}
where we are ensured that the term $(t^{-b(n)}g_{2,n}t^{b(n)})$ remains bounded as $g_n$ satisfies condition $(*)$. 
Since $\impt\subset Z_{G_{\nu}}(t)$ we may further simplify \eqref{quasiphi2'} and get 
\begin{align*}\phi_2(u)=\lim_{n\rightarrow\infty}\phi_{2,n}(u)=&\lim_{n\rightarrow\infty}t^{b(n)}z(u\cdot(t^{-b(n)}g_{2,n}t^{b(n)}))t^{-b(n)}\\ =&\lim_{n\rightarrow\infty}z(u\cdot(t^{-b(n)}g_{2,n}t^{b(n)}))\end{align*} 
Here $z(\cdot)$ is as in \eqref{decomposition} in Section~\ref{secquasi}. 
Passing to a subsequence we may assume by condition $(*)$ that $t^{-b(n)}g_{2,n}t^{b(n)}\rightarrow h\in W^-(t).$ Since by the assumption in case 2 the maps $\phi_i$ for $i=1, 2$ are non-constant we are guaranteed that $h\neq e.$  We will next get more information about $h$.

Recall now that we chose the sequence $\{g_n\}$ so that it satisfies the property (iii) in Lemma~\ref{sequ}. In particular by the remark proceeding loc.\ cit.\  we have; $\ell^-(w^-(g_{2,n})_\gamma)-\ell^-(w^-(g_{2,n})')$ tends to $+\infty$ as $n$ tends to infinity. This and the fact that $t^{-b(n)}g_{2,n}t^{b(n)}\rightarrow h\in W^-(t)$ give $h=h_{\gamma}\in U_{-\gamma},$ where the notation is as in Lemma~\ref{sequ}.

We now consider $\phi_2(u_{\gamma}(y))$ where $y\in K_{\nu}^k$. Using the arguments and the observations\footnote{As $e$ is in the domain of all rational functions considered the same is true for an open neighborhood of $e$, which in turn shows that the rational functions considered are also defined on a Zariski-open dense subset of $U_\gamma$.} above we have for any $y\in K_{\nu}^k$ that
\begin{equation}\label{special-shape}
 \phi_2(u_{\gamma}(y))=z(u_{\gamma}(y)h).
\end{equation}
As before we denote by $G_{\gamma}$ the algebraic group generated by $U_\gamma$ and $U_{-\gamma}$ which is the group of $K_{\nu}$-points of an algebraic group whose $K_{\nu}$-rank is one. 
Observe that $\phi_2$ is a $U$-quasiregular map. Since under our assumption it is not constant it will be unbounded in $G/U.$ Now due to the formula in \eqref{special-shape} we can find some $y$ such that $\tp=\phi_2(u_\gamma(y))$
is of class $\mathcal{A}'$.

As $\gamma$ is the highest root we get that $U\subset Z_G(\tp)W^+(\tp)$.  Now consider $h=\phi_1^Z(u_\gamma(y))$ (after possibly changing $y$ slightly to make this expression well defined). By the above $h\in C_G(U)$. Since $C_G(U)/C_U(U)$ is finite (equal to the image of the center of $G$) and $U_\gamma$ is connected, we get that $h\in C_U(U)$. Notice that $C_U(U)$ is just an affine space and $C_U(U)$ splits under the action of $\tp$ into two subspaces, namely $C_U(U)\cap C_U(\tp)$ and a $\tp$-invariant complement.
Using this we may find some $h'\in C_U(U)\cap C_U(\tp)$ and $z\in C_U(U)$ with $\tp h=z\tp z^{-1}h'$ where $h'$ and $z\tp z^{-1}$ commute. Now after raising to a suitable power of the characteristic of the field, i.e. the order of $h'$, we have that $\s=({}^z\tp,\tp)$ leaves $\mu$ invariant as required.    
\end{proof}

Let us fix some further notations, as before $\mu$ is an ergodic joining for the action of $U$. For our fixed place $\nu\in S$ we define
for any automorphisms $\tau$ of $\Gnu$ the subgroup
$$\Gnu=({}^{\tau}\Delta(G_{\nu}))\prod_{\omega\neq\nu}(G_{\omega}\times G_{\omega}).$$
For future references we also define for $z\in G$ the subgroups ${}^{z}\Delta(G_{\nu})={}^{(z, e)}\Delta(G_{\nu})$ and ${}^{z}H(\nu)={}^{\tau}H(\nu),$
where $\tau$ is the inner automorphism defined by $(z,e)$.

Assume $\tp$ is an element of class $\mathcal{A}'$ and $z\in Z_{G_{\nu}}(U)$ such that $\s=({}^{z}\tp,\tp)$ leaves the measure $\mu$ invariant. 
Note that $\s\in\nhn$ and $\n\subset M_{\s}W^+(\s).$ We let $\n_{\s}^+=W^+(\s)\cap\n.$ Define 
$$\mathcal{F}(\s)=\{g\in H\h|\h\n_{\s}^+g\subseteq \overline{W^-(\s)M_{\s}\n_{\s}^+}\}$$
which in fact is a subgroup of $H.$  Let $\tau$ be the inner automorphism induced by $(z,e)$ where $z\in Z_{G_{\nu}}(U)$ is as in definition of the element $\s.$
As $\Gnu$ is a group containing $\n_\s^+$ and contained in $\overline{W^-(\s)M_{\s}\n_{\s}^+}$,
we clearly have $\Gnu\subseteq\mathcal{F}(\s)$. Moreover, due to the maximality of $\Gnu$ in $H,$ it is easy to see that $\mathcal{F}(\s)= \Gnu.$  Define $\n_{\s}^-=\mathcal{F}(\s)\cap W^-(\s).$ Note that $\s$ normalizes both $\n_{\s}^+$ and $\n_{\s}^-.$

We again make use of the description of $\s$ and the description of $\nhn$ in $H$ and notice that one has $\n_{\s}^+=\Delta(W^+_G(\tp))$ and $\n_{\s}^-={}^z\Delta(W^-_G(\tp)).$ As in Section~\ref{secquasi} we let $\mathcal{V}_{\s}^{\pm}$ be the cross-section for $\n_{\s}^{\pm}$ in $W^{\pm}(\s)$ respectively, as defined in loc.\ cit. The functions $\ell^{\pm}$ there will be needed here too.

\begin{remark}\label{maximal}
Note that if $\mu$ is invariant under $(x_{1n},x_{2n})\in W^+(\s)$ where $x_{1n}x_{2n}^{-1}\rightarrow\infty$ in $G,$ then $\mu=m_1\times m_2.$ This follows if one argues as in case 1 of Proposition~\ref{torusinv}. (This says if $\mu\neq\haar$ then $\n^+$ has maximum dimension in the class of split $K_{\nu}$-algebraic subgroups of $W^+(\s)$ which leave $\mu$ invariant.)
\end{remark}

The following is closely related to \cite[Proposition 8.3]{MT}.

\begin{prop}\label{boundedbelow}
Let the notations and conventions be as above. Then at least one of the following two cases holds:
\begin{itemize}
\item[(i)] $\mu=\haar.$ 
\item[(ii)] For every $\vare>0$ there exists a compact subset $\Omega_{\vare}$ of $X$ with $\mu(\Omega_{\vare})>1-\vare$ such that if $\{g_n\}\in H\setminus N_H(\n_{\s}^+)$ is a sequence with $g_n\rightarrow e$ and $g_n\Omega_{\vare}\cap\Omega_{\vare}\neq\emptyset$ for every $n,$ then the sequence $\{\ell^-(v^-(g_n))-\ell^-(u^-(g_n))\}$ tends to $-\infty.$
\end{itemize}
\end{prop}
  
\begin{proof}
There is nothing to prove if (i) above holds, so assume $\mu\neq\haar.$ Define $R=\langle\s,\n\cap M_{\s}\rangle,$ the group generated by $\s$ and $\n\cap M_{\s}.$ Note that $\mu$ is $R$-ergodic thanks to generalized Mautner lemma~\cite[Lemma 3]{Mar1}. We fix an ergodic decomposition $\mu=\int_Y\mu_yd\sigma$ of $\mu$ for $\n_{\s}^+.$ As before for any $x\in X$ we let $y(x)$ denote the corresponding point in $(Y,\sigma).$ As $R$ normalizes $\n_{\s}^+$ we have some $R$-action on $Y,$ which in fact is a factor, i.e.\ $\mu_{y(gx)}=\mu_{gy(x)}$ for every $g\in R.$ Note also that for $\sigma$-almost every $y$ we have $\pi_{i*}\mu_y=m_i.$ This is to say for $\mu$ almost every $x\in X$ the measure $\mu_{y(x)}$ is an ergodic joining for $W^+(\tp)$ on $X.$  

We say $(\dagger)$ holds for $x\in X$ if there exist $(x_{1n},x_{2n})\in W^+(\s),\h n\in\bbn$ where $x_{1n}x_{2n}^{-1}\rightarrow\infty$ in $G,$ and $\mu_{y(x)}$ is invariant under $\{(x_{1n},x_{2n})\}$.

We claim the set of $x\in X$ for which $(\dagger)$ holds is a null set. To see the claim note that Remark~\ref{maximal} says that if $(\dagger)$ holds for some $x\in X$ then $\mu_{y(x)}=\haar,$ so we have 
$$\Upsilon=\{x\in X\h|\h(\dagger)\h\mbox{holds for}\h x\}=\{x\in X\h|\h\mu_{y(x)}=\haar\}$$ 
As a result $\Upsilon$ is a measurable set and is invariant under the action of $R.$ However $\mu$ is $R$ ergodic so either $\mu=\haar$ which we are assuming not to be the case or the set $\Upsilon$ is a null set as we wished to show.

This says we may find a compact set of uniform convergence $\Omega_{\vare}\subset X\setminus\Upsilon$ for $\n_{\s}^+,$ with $\mu(\Omega_{\vare})>1-\vare.$ Now assume $\{g_n\}\in H\setminus N_H(\n_{\s}^+),$ such that $g_n\rightarrow e$ and $g_n\Omega_{\vare}\cap\Omega_{\vare}\neq\emptyset$ but $\{\ell^-(v^-(g_n))-\ell^-(n^-(g_n))\}$ is bounded from below. We let $x_n\in\Omega_{\vare}$ such that $g_nx_n\in\Omega_{\vare}$ and $x_n\rightarrow x\in\Omega_{\vare}.$ One then constructs the quasi regular map $\phi$ corresponding to this sequence $\{g_n\}$ and $\s.$ The basic lemma says $\mu_{y(x)}$ is invariant under $\mbox{Im}(\phi).$ The construction of $\phi$ and our assumption on $\{g_n\}$ thanks to Proposition~\ref{star} says that $\mbox{Im}(\phi)\subset W^+({}^z\tp)\times\{e\}.$ However Proposition~\ref{normalizer} gives $\mbox{Im}(\phi)$ is not in $C\n_{\s}^+$ for any compact set $C$ of $H.$ All these put together using Remark~\ref{maximal} contradict the fact $x\in X\setminus\Upsilon.$ This finishes the proof of the proposition.  
\end{proof} 

\begin{cor}\label{displacement}(cf.~\cite[Corollary 8.4]{MT})\\
Let the notations be as in the Proposition~\ref{boundedbelow} and that $\mu\neq m_1\times m_2$. Then there exists a subset $\Psi$ in $X$ with $\mu(\Psi)=1$ such that $W^-(\s)x\cap\Psi\subset\n_{\s}^-x,$ for every $x\in\Psi.$  
\end{cor}

\begin{proof}
The proof follows the same lines as in the proof of \cite[Cor.~8.4]{MT} thanks to Proposition~\ref{boundedbelow} above.
\end{proof}

The following is an important application of entropy arguments which was proved in~\cite[Sect.~9]{MT}. Let us point out again that \cite{MT} assumes the characteristic to be zero, however this assumption is not needed here as one would imagine thanks to the geometric nature of entropy arguments. Here we have $X$ is as before and $\sigma$ is any probability measure on $X.$

\begin{thm}\label{entropy}(cf.~\cite[Thm.~9.7]{MT})\\
Assume $\s$ is an element from class $\mathcal{A}'$ which acts ergodically on the measure space $(X,\sigma).$ Let $V$ be a $K_{S}$ subgroup of $W^-(\s)$ normalized by $\s.$ Put $\alpha=\alpha(\s^{-1},V).$
\begin{itemize}
\item[(i)] If $\sigma$ is $V$-invariant, then $h(s,\sigma)\geq\log_2\alpha.$
\item[(ii)] Assume that there exists a subset $\Psi\subset X$ with $\sigma(\Psi)=1$ such that for every $x\in\Psi$ we have $W^-(\s)x\cap\Psi\subset V^-x.$ Then $h(\s,\sigma)\leq\log_2(\alpha)$ and the equality holds if and only if $\sigma$ is $V$-invariant.
\end{itemize}   
\end{thm} 

After these preparations we now have all the ingredients required to finish the proof of the classification of joinings.

\begin{proof}[Proof of Theorem~\ref{joining}]
 We again assume $\mu\neq m_1\times m_2.$

\vspace{1mm}
\textbf{\textit {Step 1:}} $\mu$ is invariant under $\n_{\s}^-.$

Thanks to Corollary~\ref{displacement} there exists a full measure subset $\Psi\subset X$ such that $W^-(\s)x\cap\Psi\subset\n_{\s}^-x$ for every $x\in\Psi.$ Let $\mu=\int_Y\mu_yd\sigma$ be an ergodic decomposition of $\mu$ into $\langle\s\rangle$-ergodic components. Now thanks to Mautner's lemma~\cite[Lemma 3]{Mar1}
every $\mu_y$ is $\n_{\s}^+$-invariant. Also $\mu_y(\Psi)=1$ for $\sigma$ almost every $y.$ Let $y\in Y$ be such a point. As $h(\s,\mu_y)=h(\s^{-1},\mu_y)$ Theorem~\ref{entropy} above gives 
$$\log_2\alpha(\s,\n_{\s}^+)\leq h(\s,\mu_y)\leq\log_2\alpha(\s^{-1},\n_{\s}^-)$$
However $\alpha(\s^{-1},\n_{\s}^-)=\alpha(\s,\n_{\s}^-)^{-1}$. Note that since $\mathcal{F}(\s)$ is a semisimple group it is in particular unimodular. We have  
$$\alpha(\s,\mathcal{F}(\s))=\alpha(\s,\n_{\s}^+)\alpha(\s,\n_{\s}^+)=1$$ 
Which gives 
$$h(\s,\mu_y)=\log_2\alpha(\s^{-1},\n_{\s}^-)$$ 
Now Theorem~\ref{entropy} (ii) gives $\mu_y$ is $\n_{\s}^-$-invariant for any such $y$. As this was a full measure set with respect to $\sigma,$ we get $\mu$ is invariant under $\n_{\s}^-.$

\vspace{1mm}
\textbf{\textit {Step 2:}} $\mu$ is invariant under ${}^{\tau}\Delta(G_{\nu})$ for some $\tau$ as in the Theorem~\ref{joining}.

This follows from the description of $\n_{\s}^+$ and $\n_{\s}^-$ given above, let us recall that our earlier observations said $\n_{\s}^+=\Delta(W^+(\tp))$ and $\n_{\s}^-= {}^{(z,e)}\Delta(W^-(\tp)).$ Now step 1 above says $\mu$ is invariant under $\n_{\s}^-$ and by our assumption we have the invariance under $\n_{\s}^+$ so $\mu$ is invariant under $\langle\n_{\s}^+,\n_{\s}^-\rangle={}^{(z,e)}\Delta(G_{\nu}),$ where the latter follows as we assumed that $G_\nu$ is connected, simply connected, almost simple (see for example~\cite[Theorem 2.3.1]{Mar5}).

\vspace{1mm}
\textbf{\textit {Step 3:}} Completion of the proof of Theorem~\ref{joining}.

Assume first that for some $x\in X$ we have $\mu({}^{z}H(\nu)x)>0$. As this set is $\n$ invariant and $\mu$ is $\n$-ergodic, we then have $\mu({}^{(z,e)}H(\nu)x)=1.$ Note that this is not necessarily a closed subset of $X.$  Now let ${}^{(z,e)}H(\nu)_x$ denote the stabilizer of $x$ in ${}^{(z,e)}H(\nu),$ this is a discrete subgroup of ${}^{(z,e)}H(\nu)$ and we may view $\mu$ as a measure on ${}^{(z,e)}H(\nu)/{}^{(z,e)}H(\nu)_x$. The measure $\mu$ is  ${}^{(z,e)}\Delta(G_{\nu})$-invariant and ergodic and ${}^{(z,e)}\Delta(G_{\nu})$  is a normal subgroup of ${}^{z}H(\nu)$. Therefore,  Lemma~\ref{normal-unimodular} guarantees that $\mu$ is the $\Sigma=\overline{{}^{(z,e)}\Delta(G_{\nu}){}^{(z,e)}H(\nu)_x}$-invariant measure on a closed $\Sigma$-orbit on ${}^{(z,e)}H(\nu)/{}^{(z,e)}H(\nu)_x$. However, this also implies that $\mu$ is the $\Sigma$-invariant measure on a closed $\Sigma$-orbit on $X$. We now study the structure of $\Sigma$ in more details. Let us write $x=g_1\Gamma_
 1\times g_2\Gamma_2$ and let $\Sigma_0=(g_1, g_2)^{-1}\Sigma(g_1, g_2).$ We showed that the orbit of $\Sigma_0$ from $(e, e)$ is closed and has a $\Sigma_0$-invariant probability measure on it. Hence 
\begin{equation}~\label{sigma}
  \Lambda=\Sigma_0\cap\Gamma_1\times\Gamma_2=\{(\gamma_1,\gamma_2)\in\Gamma_1\times\Gamma_2\h|\h \gamma_1=g_2g_1^{-1}z\gamma_2z^{-1}g_1g_2^{-1}\}
\end{equation}
 is a lattice in $\Sigma_0.$  Let $\mathbb{A}$ and $\mathbb{B}$ be the Zariski closure of $\Lambda$ and $\Sigma_0$ respectively. Clearly $\bba\subseteq\bbb.$ Let $A=\bba(K_S)$ and $B=\bbb(K_S).$ We claim that $A=B.$ This is a version of Borel density theorem. Let us recall the proof here.
We consider the natural map
$$\iota:{}^{(g_1^{-1}z,g_2^{-1})}H(\nu)/{}^{(g_1^{-1}z,g_2^{-1})}H(\nu)_x\rightarrow {}^{(g_1^{-1}z,g_2^{-1})}H(\nu)/A$$
Let $\mu_1$ be the push-forward of $\mu.$ Now $\mu_1$ is $\u$-ergodic on a $K_S$-variety so by lemma~\ref{bz-measure} we have this measure is concentrated on a single point which is to say $\Sigma\subseteq A.$ The claim is proved.

Recall now that $\Gamma_i$'s are arithmetic, more precisely after passing to a finite index subgroup $\Gamma_i$'s are the $S$-integer points of algebraic groups defined over the global field. Also recall that 
$$\Sigma_0=\overline{{}^{(g_{1\nu}^{-1}z,g_{2\nu}^{-1})}\Delta(G_{\nu}){}^{(g_1^{-1}z,g_2^{-1})}H(\nu)_x}$$ 
So $p_{\nu}(A)=p_{\nu}(B)={}^{(g_{1\nu}^{-1}z,g_{2\nu}^{-1})}\Delta(G_{\nu})$ where $p_{\nu}$ is the projection onto the $\nu$-component of $H$. We utilize~(\ref{sigma}) above and have $p_{\nu}(\Delta(\Gamma_1\cap{}^{g_2^{}g_1^{-1}z}\Gamma_2))$ is Zariski dense in ${}^{(g_{1\nu}^{-1}z,g_{2\nu}^{-1})}\Delta(G_{\nu}).$ But this says $\Gamma_1$ and ${}^{g_2^{}g_1^{-1}z}\Gamma_2$ give the same global structure to $G_{\nu}.$ Hence $\Gamma_1$ and ${}^{g_2^{}g_1^{-1}z}\Gamma_2$ have the same global structure. Thus $\Lambda$ is commensurable to ${}^{(g_1^{-1}z,g_2^{-1})}\Delta(\Gamma_2).$ Now since $\bbg$ is simply connected and $G_{\nu}$ is not compact it follows from strong approximation Theorem (see for example~\cite[chapter 2, section 6]{Mar5}) that $\Sigma_0={}^{(g_1^{-1}z,g_2^{-1})}\Delta(G)=B=A,$ which gives the second possible conclusion of the theorem.

Hence we may assume $\mu({}^{z}H(\nu)x)=0$ for every $x\in X.$  Let $\n^-={}^{(z,e)}\Delta(U^-)$ where $U^-$ denotes the maximal unipotent subgroup of $G_{\nu}$ opposite to $U.$  Note that $\mu$ is invariant and ergodic for ${}^{z}\Delta(G_{\nu})$ as a result it is invariant and ergodic for both $\n$ and $\n^-.$ Let $\Omega_{\vare}$ be a compact set of uniform convergence for the action of both $\n$ and $\n^-,$ which satisfies $\mu(\Omega_{\vare})>1-\vare$ for some ``small" $\vare>0.$ Let
$$
M=(G_{\nu}\times\{e\})\times\prod_{\omega\neq\nu}(G_{\omega}\times G_{\omega})
$$
Note that $M$ is a cross-section for ${}^{z}\Delta(G_{\nu})$ in H.
Since $\mu({}^{z}H(\nu)x)=0$ we may argue as in~\cite[Lemma 3.3]{MT} and find a sequence $\{h_n\}\subset M\setminus\left(\prod_{\omega\neq\nu}(G_{\omega}\times G_{\omega})\right)$ such that $h_n\rightarrow e$ and $h_n\Omega_{\vare}\cap\Omega_{\vare}\neq\emptyset.$  
We play the same old game again and construct the $\n$-quasiregular map $\phi$ with respect to $h_n.$ Our measure $\mu$ is invariant under $\mbox{Im}(\phi)$ by the basic lemma. We are obviously in case (ii) of Proposition~\ref{star} so by loc.\ cit.\  $\mbox{Im}(\phi)\subset W^+(s).$ This says in any case we have $\mbox{Im}(\phi)\subset U\times\{e\}.$ The desired unboundedness of $\mbox{Im}(\phi)$ is guaranteed by Proposition~\ref{normalizer}. Then Remark~\ref{maximal} gives $\mu=\haar.$ Theorem~\ref{joining} is proved. 
 \end{proof}


\section{Linearization}\label{seclinearization}

In this section we state and prove the linearization technique for the positive characteristic case. We do this more generally than needed in this paper with the hope that it will be useful in the future.

Let  $K$ be a global field of positive characteristic and let $S$ be a finite set of places of $K.$ As usual $K_S=\prod_{\nu\in S}K_{\nu}$ and $\mathcal{O}(K)=\mathcal{O}_S$ the ring of $S$-integers. 
Let $\bbg$ be a $K$-group and let $$G=\prod_{\nu\in S}\bbg(K_{\nu})$$
Let $\Gamma$ be an arithmetic lattice in $G$ commensurable to $\bbg(\mathcal{O}_S).$ If $k$ is a local field and $A\subset k^d$ we let $|A|$ denote the Haar measure of $A.$

The following simple lemma is a consequence of the Lagrange interpolation and is an important property of polynomials over a local field. We refer to~\cite{KM},~\cite{KT}, \cite {Gh} for a discussion of polynomial like behavior. Let $\overline{k}$ denote the algebraic closure of a field $k$.

\begin{lem}\label{polygood}
Let $k$ be a local field and $p\in \overline{k}[x_1,\cdots,x_d]$
be a polynomial of degree not greater than $l$. Then there exists
$\hspace{1mm}C=C_{d, l}\hspace{1mm}$ independent of $p$, such that for any ball $B\subset k^d$ one has.
$$\Bigl|\Bigl\{x\in{B}|\hspace{1mm}
\|p(x)\|< \vare\cdot\sup_{{x}\in
{B}}\|p(x)\|\Bigr\}\Bigr|\leq C
\hspace{1mm}\vare^{\frac{1}{dl}}|{B}|.$$ 
\end{lem}

Let us fix a few notations to be used in this section without further remarks. We let $U$ be a $K_S$-split unipotent subgroup of $G$ and let $\theta$ be the left invariant Haar measure on $U,$ which of course is right invariant as well. We let $T\subset S$ be the set of places $\omega\in S$ for which $U_{\omega}\neq \{e\}.$ Throughout this section we assume there is a polynomial parametrization  $\mathbf{u}: \prod_{\nu\in T}K_{\nu}^{d_{\nu}}\rightarrow U,$ which satisfies $\mathbf{u}(\mathbf{0})=e.$ Furthermore $\ubf_{*}\la=\theta,$ where $\lambda$ is the Lebesgue measure on $\prod_{\nu\in T}K_{\nu}^{d_{\nu}}.$ We let $\bfr$ be an open ball around the origin in $K_S^d$ fixed once and for all. We assume $U\subset W^+(s)$ for some element $s\in G$ from class $\mathcal{A}.$  

\begin{definition}
A sequence of regular maps $\la_n:U\rightarrow U$ is called a sequence of {\it admissible expanding} maps iff there exists $U_0$ such that (i) $p:U\rightarrow U_0$ is a regular isomorphism (as $K_S$-varieties) with $p(e)=e,$ and (ii) $sU_0s^{-1}=U_0$ and we have $\la_n(u)=p^{-1}(s^np(u)s^{-n}).$
\end{definition}

We make one further assumption, (this is not an essential assumption and is satisfied in most reasonable cases), assume that there exists a sequence of  {\it admissible expanding} maps, ${\la_n:U\rightarrow U}.$

The following states a quantitative non-divergence theorem for the action of unipotent groups on $G/\Gamma.$ By now the theorem has a long history. G.~A.~Margulis first proved some non-quantitative version of this which he used in the proof of arithmeticity of non-uniform lattices. The ideas developed by Margulis were applied by S.~G.~Dani to prove the first quantitative version. The study did not stop there indeed later in 90's D.~Kleinbock and G.~A.~Margulis~\cite{KM} pushed the idea further and gave more precise quantifiers which was used in theory of diophantine approximation. The $S$-arithmetic version in zero characteristic was proved in~\cite{KT} and the result in positive characteristic was obtained in~\cite{Gh}.

\begin{thm}\label{nondivergence}
Let $G$ be a $K_S$-algebraic group and let $\Gamma$ be an arithmetic lattice in $G.$ Let $U$ be the $K_S$-points of a unipotent $K_S$-split subgroup of $G.$ We assume $\mathbf{u}: \prod_{\nu\in T}K_{\nu}^{d_{\nu}}\rightarrow U$ is a polynomial diffeomorphism onto $U,$ such that $\mathbf{u}(\mathbf{0})=e.$ Let $\mathcal{K}\subset G/\Gamma$ be a compact subset and $\vare>0$, then there exists a compact set $\mathcal{L}\subset G$ such that $\mathcal{K}\subset\mathcal{L}\Gamma/\Gamma$ and for any $x\in\mathcal{K}$ and $\mathbf{B}(r)\subset\prod_{\nu\in S}K_{\nu}^{d_{\nu}},$ any ball around the origin, we have $$|\{\mathbf{t}\in\mathbf{B}(r)|\mathfrak{r}(\mathbf{t})x\in\mathcal{L}\Gamma/\Gamma\}|\geq(1-\vare)|\mathbf{B}(r)|$$
\end{thm}

\begin{proof}
Thanks to arithmeticity of $\Gamma$ we may reduce the problem to the case where $G=\mbox{SL}_n(K_S)$ and $\Gamma=\mbox{SL}_n(\mathcal{O}_S).$ The result now follows from~\cite[Thm.~6.3]{KT}  and~\cite[Thm.~4.3]{Gh}, using Lemma~\ref{polygood} and the fact that $\mathbf{0}\in\mathbf{B}(r)$. 
\end{proof}

We have the following, which is essentially in~\cite[Prop.~4.2]{T}.

\begin{prop}\label{polylinear}
Let $\nu$ be a place of $K$ and let $M$ be a Zariski closed subset in $K_{\nu}^m.$ Then for any compact subset $A$ of $M,$ and any $\vare>0$  and integers $d,D$ there exists a compact subset $B$ in $M$ such that the following holds: given a neighbourhood $\mathcal{Q}_0$ of $B$ in $K_{\nu}^m,$ there exists a neighbourhood $\mathcal{Q}$ of $A$ in $K_{\nu}^m$ such that for any unipotent subgroup $U$ of $\mbox{GL}_m(K_{\nu}),$ with $\mathbf{u}: K_{\nu}^d\rightarrow U$ a polynomial diffeomorphism  from $K_{\nu}^d$ onto $U$ of degree $\leq D$ and with $\mathbf{u}(\mathbf{0})=e,$ any $a\in K_{\nu}^m\setminus\mathcal{Q}_0$ and any $B(r)\subset K_{\nu}^d$ a neighbourhood of the origin, we have $$|\{\tbf\in{B}(r)|\mathbf{u}(\tbf)a\in\mathcal{Q}\}|\leq\vare|\{\tbf\in B(r)|\mathbf{u}(\tbf)a\in\mathcal{Q}_0\}|$$
\end{prop}

\begin{proof}
The proof is identical to that of \cite[Prop.~4.2]{T} or \cite[Prop.~4.2]{DM} using Lemma~\ref{polygood}. (In \cite{DM} the variety could be defined by a single real polynomial by taking the sum of the squares, but this is not essential as the proof there shows.)
\end{proof}

We now move to the proof of linearization. Let us, following~\cite{DM} and ~\cite{T} fix the following notations and definitions. Define 
$$\mathcal{F}= \left\{\mathbb{F}\h\lvert \h\begin{array}{c}\mathbb{F} \h\mbox{is a connected}\h K\mbox{-closed subgroup of}\h\mathbb{G}\h,\\ \mathbb{F}\h\mbox{has no rational characters}\end{array}\right\}$$
We will refer to these subgroups as subgroups from class $\mathcal{F}$. For a proper subgroup $\mathbb{F}$ of class $\mathcal{F}$ we let $F=\bbp(K_{S}).$ Let $U$ be a split unipotent subgroup of $G$ as before  and let
$$X(F,U)=\{g\in G\h|\h Ug\subset gF\}$$
As it is clear from the definition $X(F,U)$ is a $K_{S}$-closed subset of $G.$ We let 
$$\mathcal{S}(U)=\bigcup_{\mathbb{F}\in\mathcal{F}, \mathbb{F}\neq\mathbb{G}}X(F,U)/\Gamma\hspace{2mm}\mbox{and}\hspace{2mm}\mathcal{G}(U)=G/\Gamma\setminus\mathcal{S}(U)$$
Following~\cite{DM} points in $\mathcal{S}(U)$ are called singular and points in $\mathcal{G}(U)$ are called generic points with respect to $U.$ Note that these are a priori different from measure theoretic generic points, however any measure theoretic generic point is generic in this new sense as well. 

The following is the main result of this section. This was proved in~\cite{DM} in the real case for one parameter subgroups.  Later it was obtained in the $S$-arithmetic setting see~\cite{MT,T}. For an account of this for the case not dealing only with one parameter groups we refer to~\cite{Sh} and also~\cite{EMS}. 

\begin{thm}\label{linearization}
Let $G,\h\Gamma$ be as in the statement of Theorem~\ref{nondivergence}. Let $\vare>0$ and let $\mathcal{K}\subset X=G\backslash \Gamma$ be a compact subset. Furthermore, let $\mathbb{F}$ be a subgroup from $\mathbb{G}$ of class $\mathcal{F}$ and assume $C=\prod_{\nu\in S}C_{\nu}\subset\prod_{\nu\in S}\mathbb{G}(K_{\nu})$ is a 
compact subset of $X(F,U).$ Then we 
can find a maybe larger compact subset $D=\prod_{\nu\in S}D_{\nu}\supset\prod_{\nu\in S}C_{\nu}$ of
$X(F,U)$ with $D_{\omega}=C_{\omega}$ for  $\omega\in S\setminus T$ such that the following holds; 
For any neighbourhood $\Phi$ of $D$ in $G$ there exists a neighbourhood $\Psi$ of $C$ such that if $x\in\mathcal{K}\setminus(\Phi\Gamma/\Gamma)$ then for any $n\in\bbn$
$$\frac{1}{\theta(\bfr)}\theta\Bigl(\Bigl\{\tbf\in\bfr|\h\la_n(\ubf(\tbf))x\in\Psi\Gamma/\Gamma\Bigr\}\Bigr)<\vare$$
\end{thm} 

Let us note that this theorem is the reason why a classification of $U$-invariant measures (Dani's measure conjecture) implies a classification of $U$-orbit closures (Raghunathan's conjecture). More precisely, if it is known that a $U$-invariant and ergodic measure must be the Haar measure of an orbit of the form $gF\Gamma/ \Gamma$ for some $g\in X(F,U)$ and $\mathbb{F}\in\mathcal{F}$, then by Theorem \ref{linearization} any $x\in\mathcal{G}(U)$ is measure-theoretic generic for the Haar measure $m_X$ on $X=G/\Gamma$ and so has dense orbit. (While those points $x\in \mathcal{S}(U)$ are generic for a smaller dimensional Haar measure). Of course, in zero characteristic the measure classification is known due to \cite{Rat5, R3} resp.\ \cite{MT}.

The rest of this section is devoted to the proof of Theorem~\ref{linearization}, the proof will follow the argument in~\cite{DM}, as it did in~\cite{T}, after setting up everything correctly there are not many difficulties arising in the body of the proof. We will however present the proof for the sake of completeness.

We may, as we will, reduce to the case $T=\{\nu\}$ is a singleton. This is a simple induction argument, for details see~\cite[Lemma 4.3]{T}.
After this, we may do one more reduction which will be helpful: As we are in positive characteristic there is an open and compact subgroup $M\subset\prod_{\omega\in S\setminus\{\nu\}}G_\omega$. As $G_\nu\times M\subset G$ is open and $\Gamma$ is a lattice, we have that $G/\Gamma$ is a finite union of $G_\nu\times M$-orbits $G_\nu\times M (e,h_i)\Gamma$ with $i=1,\ldots,\ell$. Each of these orbits is $U$-invariant and the statement of the Theorem \ref{linearization} for $G/\Gamma$ is equivalent to the statement for each of the $G_\nu\times M$-orbits. Therefore, we may assume that $G'=G_\nu\times M$ and that $\Gamma'\subset G'$ is such that $\pi_\nu(\Gamma')$ is commensurable with $\pi_\nu(\{\gamma\in\Gamma: \pi_{S\setminus\{\nu\}}(\gamma)\in M\})$. However, we may moreover assume that the set $C$ is a product $C_\nu\times M$ in which case the proposition is equivalent to its statement for $C_\nu$ and the quotient  $G_\nu/\pi_\nu(\Gamma')\simeq M\backslash G'/\Gamma'$. Here $\pi_\nu(\Gamma)$ is commensurable with $\mathbb{G}(\mathcal{O}_\nu)$ --- we note that for every place $\nu$ there is indeed a corresponding subring $\mathcal{O}_{\{\nu\}} \subset K$. In other words, the case of $S=\{\nu\}$ implies the general case. Hence we may and will assume $S=\{\nu\}$ and don't have to write the subscripts in the remainder of the section.

Let $N_{\bbg}(\bbp)$ be the normalizer of $\bbp$ in $\bbg.$ This is a $K$-closed subset of $\bbg$. We will construct a finite $K$-rational representation $\rho$ of $\bbg$ on a $K$-vector space $W$ and a vector $w_0\in W$ such that $N_\bbg(\bbp)=\{g\in \bbg: \rho(g)w_0\in \overline{K}w_0\}$ just as in Chevalley's theorem. However, as we wish to obtain additional information we record this in the following. 

\begin{lem}\label{representation}
There is a representation $\rho:\bbg\rightarrow\bbg\mathbb{L}(\mathbb{W})$ which is rational over $K,$ a $K$-rational character $\chi$ defined on $N_\bbg(\bbp)$, and  a vector $w_0\in W=\mathbb{W}(K)$ such that:
\begin{itemize}
\item[(i)] $N_{\bbg}(\bbp)=\{g\in \bbg: \rho(g)w_0\in \overline{K}w_0\}=\{g\in \bbg: \rho(g)w_0=\chi(g)w_0\}$
\item[(ii)] The orbit $\rho(\Gamma)w_0$ is discrete and closed.
\item[(iii)] Let $\eta:\bbg\to W$ be the orbit map i.e. $\eta(g)=\rho(g)w_0$ for $g\in\bbg$. Then 
\begin{equation}\label{eta2}
\eta^{-1}(\overline{\eta(X(F, U))})=X(F, U)
\end{equation}
where $\overline{\eta(X(F, U))}$ is the Zariski closure of $X(F, U)$ in $W$.
\end{itemize}
\end{lem}

\begin{proof}
Note that as $\bbg$ is defined over $K$ the right and left regular representation of $\bbg$ on $\overline{K}[\bbg]$ is defined over $K$. Let us first recall that as $\bbp$ is $K$-closed there exists a purely inseparable extension $E$ of $K$ such that $\bbp$ is defined over $E$ (see for example in~\cite{B1}).
Now $\bbp$ is defined over $E$ so we can find $\{h_1,\cdots,h_k\}$ which are $E$-rational generators of the ideal $\mathcal{J}_0$ defining $\bbp.$ Let $q$ be a power of the Frobenius map which sends $E$ to $K.$ For all $1\leq i\leq k$ we let $f_i=h_i^q$. Now let $\mathcal{J}=\mathcal{I}\cdot\overline{K}[\bbg]$ where $\mathcal{I}$ is the ideal generated by $\{f_i\}$'s in $K[\bbg].$ Also let us define the representation $\rho'(g)$ on $\overline{K}[G]$ by $(\rho'(g)h)(x)=h(g^{-1}xg)$. 

Let $\alpha:E[\bbg]\rightarrow E[\bbg]\otimes E[\bbg]$ be the co-morphism and let $\alpha(h_i)=\sum_j \beta_j\otimes\gamma_j$ and $\alpha(\gamma_j)= \sum_\ell \theta_{j\ell}\otimes\eta_{j\ell}.$ We have $\rho'(g)h_i(x)=h_i(g^{-1}xg)=\sum_j\sum_\ell\beta_j^q(g^{-1})\eta^q_{j\ell}(g)\theta^q_{j\ell}(x).$ So there is a finite dimensional subspace of $K[\bbg]$ containing all of $\rho'(g)f_i$'s. Let $V$ be the intersection of all these dimensional subspaces. We define the $K$-rational representation $\rho_V(g)f=\rho'(g)f$ for $f\in V$. 

Let $I=V\cap\mathcal{J}.$ Note that $I$ has a $K$-basis of functions in $K[\bbg]$ consisting of elements that are $q$-th power of functions in $E[\bbg].$ We define the linear space $W=\bigwedge^{\dim I}V$ as the wedge power of $V$ with the power equal to the dimension of $I$, and define $w_0$ as the wedge product of a basis of $I$. The representation $\rho$ of $\bbg$ on $W$ is the one induced by $\rho_V$ on $V$.
With these definitions we claim that $N_\bbg(\bbp)=\{g\in\bbg: \rho(w_0)\in\overline{K_\nu}w_0\}=\{g\in \bbg: \rho_V(g)I=I\},$ which will establish (i) above. Let $g\in\bbg$ be such that $\rho_V(g)I=I$ however the zero set of $I$ and $\mathcal{J}_0$ are the same hence the action of $g$ preserves the zero set of $\mathcal{J}_0$ which is $\bbp$ hence $g\in N_\bbg(\bbp).$ For the other direction let $f\in I\subset\mathcal{J}$ then it is clear from our definition of $\rho_V$ that $\rho_V(g)f=\sum_ic_i\ell^q_i$ where $\ell^q\in I$ and $c_i$ are constants. This establishes (i) above. Note that as $\Gamma$ is arithmetic and $w_0$ is $K$-rational (ii) is immediate.

Now let $\eta:\bbg\to W$ be the orbit map $\eta(g)=\rho(g)w_0$ for $g\in\bbg$. Finally let $\mathbb{X}=\{g\in\bbg|\h Ug\subset g\bbp\}.$ This is a $K_S$-closed subset of $\bbg$ and we have $\mathbb{X}(K_S)=X(F,U).$ We want to establish (iii) above. To see this, notice first that $g^{-1}ug\in\bbp$ iff the kernel of the evaluation map at $g^{-1}ug$ contains $I$ which in turn happens iff the kernel of the evaluation map at $u$ contains $\rho_V(g)(I)$. This condition is actually a polynomial relation for the coefficients of $\rho(g)w_0$. Therefore, for every $u\in U$ there exists a polynomial $p$ on $W$ such that $g^{-1}ug\in\bbp$ iff $p$ vanishes on $\rho(g)w_0$. Varying $u\in U$ we get that there is a $K$-closed variety $\mathbb{Y}\subset W$ such that $g^{-1}Ug\subset\bbp$ iff $\eta(g)\in\mathbb{Y}$. This proves the claim.
\end{proof}

Furthermore, we define the subgroup $\bbl=\{g\in\bbg: \rho(w_0)=w_0\}$ of $N_\bbg(\bbp)$,  $L=\bbl(K_S)$,  $\Gamma_L=\Gamma\cap L$, and $\Gamma_F=\Gamma\cap N_G(F).$

We now prove the following simple but important proposition.  

\begin{prop}\label{proper}
The following map is a proper map; 
$$\vartheta:G/\Gamma_L\rightarrow G/\Gamma\times W\hspace{2mm}\mbox{\rm{where}}\hspace{2mm}\vartheta(g\Gamma_L)=(g\Gamma,\rho(g)w_0).$$ 
\end{prop}
\begin{proof}
The proof is straightforward. Let $\{g_n\Gamma_L\}$ be a sequence in $G/\Gamma_L$ such that $\vartheta(g_n\Gamma_L)$ converges, we need to show that $\{g_n\Gamma_L\}$ converges as well. From the definition and our assumption we have $(g_n\Gamma,\rho(g_n)w_0)\rightarrow(g\Gamma,v)\in G/\Gamma\times W$. This says we can write $g_n=x_n\gamma_n$ where $\gamma_n\in\Gamma$ and $x_n\rightarrow g.$ Now recall that $\rho(g_n)w_0=\rho(x_n\gamma_n)w_0\rightarrow v$ so we have $\{\rho(\gamma_n)w_0\}$ converges to $\rho(g^{-1})v.$ However arithmeticity gives that $\Gamma w_0$ is discrete in $W$ so for large enough $n$ one has $\rho(\gamma_n)w_0=\rho(\gamma) w_0.$ By definition of $\Gamma_L$ this gives $g_n\Gamma_L\rightarrow g\gamma\Gamma_L,$ as we wanted to show. 
\end{proof}



Let $A$ be a subset of $G$, a point $x\in A$ is called a {\it point of} $(F,\Gamma)$-{\it self-intersection} for $A$ if there exists $\gamma\in\Gamma\setminus\Gamma_F$ such that $x\gamma\in A.$ This notion was first used in~\cite{DM} also it was used in the same way in~\cite{T}. Note that for a compact set $A$ the set of $(F,\Gamma)$-self-intersections is closed.  The following are versions of \cite[Prop.~3.3]{DM} resp.\ \cite[Cor.~3.5]{DM} and \cite[Prop.~4.9]{T}.

\begin{prop}\label{self-int1}
 For any $\mathbb{F}\in\mathcal{F}$ we have that the self-intersections of $X(F,U)$ are contained in the union of $X(F',U)$ for subgroups $\mathbb{F}'\in\mathcal{F}$ of smaller dimension. In fact, for every $g,g\gamma\in X(F,U)$ with $\gamma\notin\Gamma_F$ the subgroup
 $\mathbb{F}'\in\mathcal{F}$ with $g\in X(F',U)$ can be chosen to depend only on $\mathbb{F}$ and $\gamma$.
\end{prop}

\begin{proof}
 Let $\gamma\in\Gamma\setminus\Gamma_F$. Then $\mathbb{F}\cap\gamma\mathbb{F}\gamma^{-1}$ is a $K$-subgroup of $\mathbb{F}$ of strictly smaller dimension as $\mathbb{F}$ is connected and $\gamma\notin\Gamma_F$. Let $\mathbb{F}'\subset (\mathbb{F}\cap\gamma\mathbb{F}\gamma^{-1})^\circ$ be the common kernel of all $K$-characters of $ (\mathbb{F}\cap\gamma\mathbb{F}\gamma^{-1})^\circ$. Then $\mathbb{F}'\in\mathcal{F}.$
 
 Now suppose $g,g\gamma\in X(F,U)$, then $g^{-1} U g\subset F$ and $g^{-1} U g\subset \gamma F \gamma^{-1}$ by definition of $X(F,U)$.
 As $U$ is a connected unipotent  subgroup we have  $g^{-1} U g\subset F'$.
 Hence $g\in X(F',U)$ again by definition.
\end{proof}

\begin{prop}\label{self-intersection}
Let $D\subset G$ be a compact subset of $G$ and let $\mathcal{Y}$ be the set of all $(F,\Gamma)$-self-intersection points in $D.$ Then for any compact subset $D'$ of $D\setminus\mathcal{Y}$ there exists an open neighborhood $\Omega$ of $D'$ in $G$ such that $\Omega$ does not contain any point of $(F,\Gamma)$-self-intersection.
\end{prop}

\begin{proof}
Assume to the contrary so there exists a decreasing sequence $\Omega_n$ of open neighborhoods of $D'$ such that $\bigcap_n\Omega_n=D'$ and that for any natural number $n$ there are elements $g_n, g_n'\in\Omega_n$ where $g_n=g_n'\gamma_n$ and $\gamma_n\not\in\Gamma_F$.  Passing to a subsequence if necessary we assume $\{g_n\}$ and $\{g_n'\}$ converge. This gives $\gamma_n$ converges, thus for large enough $n$ we have $\gamma_n=\gamma\notin\Gamma_F$. Now if $g_n\rightarrow g,$ then we have $g,g\gamma\in D'.$ This says $g\in\mathcal{Y},$ which contradicts the fact $g\in D'.$ 
\end{proof}

\begin{proof}[Proof of Theorem~\ref{linearization}.] Let the notations and assumptions be as in the statement of Theorem~\ref{linearization}. The proof will be by induction on the $\dim\bbp.$ Note that there is nothing to prove when $\dim\bbp=0.$ Recall that we assume $S=\{\nu\}$, i.e.\ that $
G= G_\nu$.

So suppose $C\subset X(F,U)$ is compact. We may suppose $C\subset\mathcal{K}$.
 We apply Proposition~\ref{polylinear} to the compact subset $A=\rho(C) w_0\subset M=\overline{\rho(X(F_\nu,U))w_0}$, which defines for us a compact subset $B\subset M$. We may assume $A\subset B$. Also use Theorem~\ref{nondivergence} to find a compact subset $\mathcal{L}$ of $G$ with the properties $\mathcal{K}\subset\mathcal{L}\Gamma/\Gamma$ and that for any $x\in\mathcal{K}$
\begin{equation}\label{linearizationnond}
\frac{1}{\theta(\bfr)}\theta(\{\mathbf{t}\in\bfr|\la_n(\ubf(\mathbf{t}))x\in\mathcal{L}\Gamma/\Gamma\})\geq1-\vare.
\end{equation}
From Proposition~\ref{proper} and (\ref{eta2}) it follows that there exists a compact subset 
$D_0\subset X(F,U)$ such that 
$$
  \vartheta^{-1}(\mathcal{L}\Gamma/\Gamma\times B)= D_0\Gamma_L/\Gamma_L\h\h\h\mbox{and}\h\h\h D_0L\cap\mathcal{L}\Gamma\subset D_0\Gamma_L.
$$
We may suppose $D_0\supset C$.
As the set $\mathcal{Y}$ of all $(F,\Gamma)$-self-intersections of $D_0$ may be non-empty (and we need to control these later) we will define $D$ only after discussing $\mathcal{Y}$.

As $D_0\subset X(F,U)$ we know by Proposition \ref{self-int1} that $\mathcal{Y}$ belongs to the union of the $X(F',U)$ for subgroups $\mathbb{F}'\in\mathcal{F}$ of strictly smaller dimension. Moreover, as $D_0$ is compact and $\Gamma$ is discrete, there are only finitely many 
$\gamma\in\Gamma$ with $D_0\cap D_0\gamma\neq\emptyset$. Therefore, again by Proposition~\ref{self-int1} we have 
$\mathcal{Y}\subset \bigcup_{i=1}^\ell X(F_i',U)$
for finitely many $\mathbb{F}'_1,\ldots,\mathbb{F}'_\ell\in\mathcal{F}$. For each $i$ we define $C_i=D_0\cap X(F_i',U)$
and apply the inductive hypothesis to obtain the compact subsets
$D_i\subset X(F_i',U)$ satisfying the conclusion of the theorem with $\epsilon$ replaced by $\frac\epsilon\ell$.

We define $D=D_0\cup\bigcup_{i=1}^\ell D_i$.
To show that $D$ satisfies the property in the theorem suppose $\Phi$ is a neighborhood of $D$ in $G$. 
We will now work towards the definition of the neighborhood $\Psi$ of $C$.
As $\Phi$ is a neighborhood of $D_i$ for each $i=1,\ldots,\ell$,  we can apply the inductive hypothesis to find open
neighborhoods $\Psi_i$ of $C_i$. This shows that as long as $x\in\mathcal{K}\setminus(\Phi\Gamma/\Gamma)$ then for any $n\in\mathbb{N}$
we have
\begin{equation}\label{lower-dim-case}
 \frac{1}{\theta(\bfr)}\theta\Bigl(\Bigl\{\tbf\in\bfr|\h\la_n(\ubf(\tbf))x\in\bigcup_{i=1}^\ell\Psi_i\Gamma/\Gamma\Bigr\}\Bigr)<\vare
\end{equation} 

Removing $\Psi'=\bigcup_{i=1}^\ell \Psi_i$ from $D$ we obtain, by construction of the sets $C_i$, a compact subset 
$D'=D\setminus \Psi'\subset D\setminus\mathcal{Y}$. By Proposition~\ref{self-intersection} there exists
a neighborhood $\Omega\subset\Phi$ of $D'$ that also has no $(F,\Gamma)$-self-intersections. This gives a neighborhood $\Omega\cup \Psi'$ of $D$ and so in particular of $D_0$. Recall that $D_0\Gamma_L/\Gamma_L=\vartheta^{-1}(\mathcal{L}\Gamma/\Gamma\times B)$. We claim there is a neighborhood $\mathcal{Q}_0$ of $B$ such that 
\begin{equation}\label{linearproper}
 (\Omega\cup\Psi')\Gamma_L/\Gamma_L\supset\vartheta^{-1}(\mathcal{L}\Gamma/\Gamma\times \mathcal{Q}_0).
\end{equation} 
This is again a simple compactness argument relying on the properness of $\vartheta$ in Proposition \ref{proper}. By construction of $B$ (which relied on Proposition \ref{polylinear}) there exists now a neighborhood $\mathcal{Q}$ of $A=\rho(C) w_0$.
By continuity of the representation $\rho$ there exists a neighborhood $\Psi$ of $C$ such that $\rho(\Psi)w_0\subset\mathcal{Q}$. 
We claim that $\Psi$ satisfies the statement in the theorem for the given $\Phi$ and $(2+f\kappa)\vare$ instead of $\vare$, where $\kappa$ is the number of roots of unity in $K$ and $f$ is a constant depending on $U$ (more precisely the dimension and the residue field of $K_\nu$).

Let us indicate how the roots of unity will enter the estimate. Recall that $N_{\mathbb{G}}(\mathbb{F})$ equals the stabilizer of $\overline{K}w$
and that $w$ is $K$-rational. Hence there exists a $K$-character $\chi$ on $N_{\mathbb{G}}(\mathbb{F})$ such that $\rho(g)w_0=\chi(g)w_0$ for all $g\in N_G(F)$. As $\Gamma$ is commensurable with $\mathbb{G}(\mathcal{O}_\nu)$ and $\Gamma_F$ is a subgroup, it follows that $\rho(\Gamma_F)\subset K^\times$ is a bounded subgroup of $K_\omega^\times$ for all places $\omega\neq\nu$. However, this shows that $\rho(\Gamma_F)$ is contained in the group of roots of unity of $K$.

To show the theorem suppose $x=g\Gamma\in\mathcal{K}\setminus(\Phi\Gamma/\Gamma)$ and 
fix some integer $n\in\mathbb{N}$. Recall the union $\Psi'=  \bigcup_{i=1}^\ell\Psi_i$ of the neighborhoods
that was obtained via the inductive hypothesis. We define the following sets
$$\mathfrak{B}^{(1)}=\Bigl\{\tbf\in\bfr|\h\la_n(\ubf(\tbf))\not\in\mathcal{L}\Gamma/\Gamma\h\mbox{or}\h\la_n(\ubf(\tbf))\in\Psi'\Gamma/\Gamma\Bigr\}$$
and 
$$\mathfrak{B}^{(2)}=\Bigl\{\tbf\in\bfr|\h\la_n(\ubf(\tbf))\in(\Psi\Gamma/\Gamma\cap\mathcal{L}\Gamma/\Gamma)\setminus(\Psi'\Gamma/\Gamma)\Bigr\}.$$
Note that $\bfr^{(1)}\cup\bfr^{(2)}\supset\{\tbf\in\bfr|\h\la_n(\ubf(\tbf))\in\Psi\Gamma/\Gamma\}.$ Furthermore using (\ref{linearizationnond}) and (\ref{lower-dim-case}) above we have 
\begin{equation}\label{b1bound}\theta(\bfr^{(1)})\leq2\vare\theta(\bfr)\end{equation}
To finish the proof we need to give a similar upper bound for $\theta(\bfr^{(2)})$. 

For this notice first that for all $\gamma\in\Gamma$ we have $\rho(g\gamma) w\not\in\mathcal{Q}_0.$ Assume the contrary, then as $g\Gamma\in\mathcal{L}\Gamma/\Gamma$ we have $g\gamma\Gamma_L\in\vartheta^{-1}(\mathcal{L}\Gamma/\Gamma\times\mathcal{Q}_0)\subset\Omega\Gamma_L/\Gamma_L$ which implies $x=g\Gamma\in\Phi\Gamma/\Gamma$ which contradicts the assumption on $x$ and so proves the claim. 

Fix some $\gamma$ and the corresponding $q=\rho(g\gamma) w_0\in \rho(g\Gamma) w_0$. For each such $q$ we will define $\bfr_q\subset\bfr$ as follows; $\tbf\in\bfr_q$ if there is an open ball $O\subset\bfr$ such that 
$$\rho(\la_n(\ubf(O))g\gamma) w\subset\mathcal{Q}_0
$$
and there exists some
$$
\tbf'\in O\mbox{ such that }\la_n(\ubf(\tbf'))x\in\mathcal{L}\Gamma/\Gamma\setminus(\Psi'\Gamma/\Gamma).
$$ 
Note that $\bfr^{(2)}\subset\bigcup_{q}\bfr_q$, where the union
is taken over all $q\in\rho(g\Gamma)w$. For any $\tbf\in\bfr_q$ we let $\bfr_q(\tbf)$ be the largest ball in $\bfr_q$ which contains $\tbf.$
Due to the non-Archimedean property this definition makes sense as the union of two non-disjoint balls is always a ball. Also note that by definition there exists for any $\tbf\in\bfr_q$ some $\tbf'\in\bfr_q(\tbf)$ with $\la_n(\ubf(\tbf'))x\in\mathcal{L}\Gamma/\Gamma\setminus(\Psi'\Gamma/\Gamma)$.

We have the following property for the various $\bfr_q$'s which we will refer to as:\\
$(\ddagger)$ {\em Let $q\neq q'$ be in $g\Gamma w,$ such that $\bfr_q\cap\bfr_{q'}\neq\emptyset$ then $q'=\zeta q$ where $\zeta$ is a root of unity in $E$.}

To see this let $\tbf\in\bfr_q\cap\bfr_{q'},$ and let $\bfr_q(\tbf)$ (resp. $\bfr_{q'}(\tbf)$) be the largest ball in $\bfr_q$ (resp. in $\bfr_{q'}$) which contains $\tbf,$ as above. Due to the non-Archimedean property one of these balls contains the other which implies that there exists $\tbf_0\in\bfr_q(\tbf)\cap\bfr_{q'}(\tbf)$ such that $\la_n(\ubf(\tbf_0))x\in\mathcal{L}\Gamma/\Gamma\setminus(\Psi'\Gamma/\Gamma).$ The definitions give $\la_n(\ubf(\tbf_0))q=\la_n(\ubf(\tbf_0))g\gamma w$ and $\la_n(\ubf(\tbf_0))q'w=\la_n(\ubf(\tbf_0))g\gamma'w$ are in $\mathcal{Q}_0$. Hence by (\ref{linearproper}) we get  that 
$$
 \la_n(\ubf(\tbf_0))g\gamma, \la_n(\ubf(\tbf_0))g\gamma\in \Omega\Gamma_L.
$$ 
However, by construction $\Omega$ does not have any $(F,\Gamma)$-self-intersection. Thus we must have $\gamma'=\gamma\delta$ where $\delta\in\Gamma_F.$ As discussed above $\rho(\delta)w=\chi(\delta)w$ and $\chi(\delta)\in K$ is a root of unity which proves $(\ddagger)$.

Due to the maximality of the ball $\bfr_q(\tbf)$ and since it is a proper subset of $\bfr$ (due to the earlier established fact that $g\gamma w\not\in\mathcal{Q}_0$ for all $\gamma\in\Gamma$) we may apply the properties of $C,D,\mathcal{Q},\mathcal{Q}_0$ as given by Proposition \ref{polylinear} --- strictly speaking to the unique next largest ball containing $\bfr_q(\tbf)$ which is in $\theta$-measure only by a constant $c$ bigger. This gives
$$
 \theta(\bfr^{(2)}\cap\bfr_q(\tbf))\leq c\vare\theta(\bfr_q(\tbf)).
$$
We use the non-Archemidean feature  once more to say that for $\tbf$ and $\tbf'$ in $\bfr_q$ either $\bfr_q(\tbf)\cap\bfr_q(\tbf')=\emptyset$ or $\bfr_q(\tbf)=\bfr_q(\tbf')$, i.e.\ if we list these balls (without repetitions) we get a partition of $\bfr_q$. Summing over this partition we obtain
\begin{equation}\label{bqbound}\theta(\bfr^{(2)}\cap\bfr_q)\leq c\vare\theta(\bfr_q)\end{equation}

Recall that $\{\bfr_q\}$ gives a covering for $\bfr^{(2)}.$ This time we do not claim disjointness of the various sets. However, by $(\ddagger)$ the multiplicity of this cover is bounded by $\kappa$. Therefore, we get from summing (\ref{bqbound}) over all choices of $q$ the inequality 
\begin{multline*}
 \theta(\bfr^{(2)})\leq\sum_q\theta(\bfr^{(2)}\cap\bfr_q)\leq \\
 \leq c\vare\sum_q\theta(\bfr_q)\leq c\kappa\vare\theta\Bigl(\bigcup_q\bfr_q\Bigr)\leq c\kappa\vare\theta(\bfr)
\end{multline*}
This and (\ref{b1bound}) complete the proof.
\end{proof}


\section{Proof of Theorem~\ref{orbitclosure}}
Let the notation and conventions be as in the introduction and the statement of Theorem~\ref{orbitclosure}. In particular   
let $\bbg$ be a connected, simply connected, absolutely almost simple group defined over $K.$ Define $G=\prod_{\nu}\bbg(K_{\nu})$ and $H=G\times G$. Let
$$\mathcal{F}= \left\{\mathbb{F}\h\lvert \h\begin{array}{c}\mathbb{F} \mbox{ is a connected } K\mbox{-closed subgroup of}\h\mathbb{G}\times\bbg,\\ 
\mathbb{F}\h\mbox{has no}\h K\h\mbox{-rational character}\end{array}\right\}$$
If there is no fear for confusion, we will simply say $F\in\mathcal{F}.$
 
\begin{prop}\label{unipotentorbit}
Let $x=(x_1,x_2)=(g_1\Gamma_1,g_2\Gamma_2)\in X$ satisfy
that $x_1$ and $x_2$ are measure theoretically generic for the action of $U$ on $G/\Gamma_1$ and $G/\Gamma_2,$ respectively. 
Then $\n\cdot x$ is equidistributed in $(g_1,g_2)F_0\cdot\Gamma$. Here either $F_0={}^\tau\Delta(G)$ if there exists some inner automorphism $\tau$ for which $\n(g_1,g_2)\subseteq (g_1,g_2){}^\tau\Delta(G)$ and ${}^\tau\Delta(G)\in\mathcal{F}$, resp.\ $F_0=H$ if there doesn't exists such an automorphism $\tau$. 
\end{prop}

\begin{proof}
Let $\bfr$ be a fixed neighborhood of the identity in $\n$ as in Section~\ref{seclinearization}. Note that $\n$ comes equipped with natural family of admissible expansions thanks to the fact that $U$ is a horospherical subgroup of $G.$ We keep the notation $\la_n$ for this family as in Section~\ref{seclinearization}. Let $\widetilde{X}$ be the one-point compactification of $X$ if $X$ is not compact and be $X$ if $X$ is compact. For any natural number $n$ define the probability measure $\mu_n$ on $X$ by 
$$\int_X f(y)d\mu_n(y)=\frac{1}{\theta(\bfr)}\int_\bfr f(\la_n(\tbf)x)d\theta(\tbf)$$ 
where $f$ is a bounded continuous function on $X.$ As $\widetilde{X}$ is compact the space of probability measures on $\widetilde{X}$ is weak$^*$ compact. Let $\mu$ be a limit point of $\{\mu_n\}$. By identifying $\mu$ we show that there is only one limit points which in return gives convergence. It follows from nondivergence of unipotent trajectories (Theorem~\ref{nondivergence}) that $\mu$ is concentrated on $X$. Note also that thanks to polynomial like behavior of unipotent flows we have $\mu$ is $\n$-invariant. The condition  in the proposition regarding the genericity of the components of $x$ guarantees that $\mu$ is a joining for the $U$ action on $X_1$ and $X_2$. We let $\mu=\int_Y\mu_yd\sigma(y)$ be a decomposition of $\mu$ into $\n$-ergodic components. By assumption $\bbg$ is simply connected so that the action of $U$ on $X_1$ and $X_2$ is ergodic. Therefore, for a.e.\ $y\in Y$ the ergodic measure $\mu_y$ is a joining. We now consider the two cases in the proposition separately.

\textbf{\textit{Case 1.}} Suppose there exists an inner automorphism $\tau$ with $F_0={}^\tau\Delta(G)$ being $K$-closed and satisfying
$\n(g_1,g_2)\subseteq(g_1,g_2)F_0$. In this case $\mu$ is a measure on the closed set $(g_1,g_2)F_0\Gamma\subset X$. As a result we have $\mu\neq\haar$ and the same holds for almost every ergodic component $\mu_y$. Now recall from steps 2 and 3 of the proof of Theorem~\ref{joining} in Section~\ref{secjoining} that the ergodic joinint $\mu_y$ is invariant under ${}^{z_y}\Delta(G)$ for some $z_y\in Z(U)$. Since $\mu_y\neq\haar$ we have $\mu_y$ is the ${}^{z_y}\Delta(G)$-invariant measure on a closed ${}^{z_y}\Delta(G)$-orbit. However, the support of $\mu_y$ is in $(g_1,g_2)F\Gamma.$ So we have for $\sigma$-almost every $y,\h (g_1,g_2)F(g_1,g_2)^{-1}={}^{z_y}\Delta(G)$ thus $\mu$ is the Haar measure on $(g_1,g_2)F\Gamma.$

\textbf{\textit{Case 2.}} There is no inner automorphism as in case 1.
Now Theorem~\ref{joining} says that $\sigma$-almost every ergodic component $\mu_y$ is either the Haar measure on $X$ or the Haar measure on a closed orbit of a subgroup ${}^{z_y}\Delta(G)$. In the latter case we also know that $\mu_y$ is the Haar measure on $(e,g_y)F_y\Gamma$ where $F_y={}^{\tau_y}\Delta(G) \in\mathcal{F}$. As $\mathcal{F}$ is countable, we can rewrite the above ergodic decomposition as the following countable sum
\begin{equation}\label{measuredecomp}
\mu=c\haar+\sum_i\mu_{\delta_i}.
\end{equation}
Here $\haar$ is the Haar measure on $X$, $c\in[0,1]$, and $\mu_{\delta_i}$ is the restriction of $\mu$ to 
\begin{equation}\label{singular}
X_i=X({}^{\delta_i}\Delta(G),\n)\setminus\bigcup\left\{X(F',\n)|\h \begin{array}{c}F'\in\mathcal{F}, F'\subset{}^{\delta_i}\Delta(G), \\ \dim F'<\dim{}^{\delta_i}\Delta(G)\end{array}\right\}.
\end{equation}
Here $\delta_i\in G$ is a sequence such that the sequence ${}^{\delta_i}\Delta(G)\in\mathcal{F}$ contains precisely one element of each element of $\mathcal{F}$ of the form ${}^\tau\Delta(G)$.  Furthermore each $\n$-ergodic component of $\mu_{\delta_i}$ is the $h({}^{\delta_i}\Delta(G))h^{-1}$-invariant measure on the closed orbit $h \h{}^{\delta_i}\Delta(G)\hh\Gamma$ for some $h\in X_i.$ We now make use of Theorem~\ref{linearization} which gives  
$$\mu(X(F,\n))=0\h\mbox{for any}\h F\neq H$$ 
This using (\ref{singular}) and (\ref{measuredecomp}) above guarantee that $\mu=\haar.$
\end{proof}

\noindent
We can now complete the proof of Theorem~\ref{orbitclosure}. Let $x=(g_1\Gamma_1,g_2\Gamma_2)\in X$ be an arbitrary point. We want to consider the orbit $\Delta(G)\cdot x$. First note that as $U$ acts ergodically on $G/\Gamma_i$ for $i=1,2,$ there is an $m_G$-full measure subset $\mathcal{G}$ of $G$ such that $(g,g)\cdot x$ satisfies the genericity hypothesis of Proposition~\ref{unipotentorbit} for all $g\in\mathcal{G}$. For any $F={}^\tau\Delta(G)\in\mathcal{F}$ we define
$$
 \mathcal{E}_F=\{g\in\mathcal{G}|\h\n(gg_1,gg_2)\subset(gg_1,gg_2)F\}.
$$
 Note now that $\mathcal{F}$ is countable, so either there is some $g\in\mathcal{G}\setminus\bigcup_F \mathcal{E}_F$ or $m_G(\mathcal{E}_{F_0})$ is positive for some $F_0.$ If $g\in\mathcal{G}\setminus\bigcup_F \mathcal{E}_F$ then by Proposition~\ref{unipotentorbit}, we have that $(gg_1,gg_2)\Gamma$ is generic for $m_{X_1}\times m_{X_2}$ which proves that $\Delta(G)(g_1,g_2)\Gamma$ is dense. On the other hand,  if $m_G(\mathcal{E}_{F_0})$ is positive for some $F_0$ then $(g_1,g_2)^{-1}\Delta(G)(g_1,g_2)\subset F_0$ and $\Delta(G)\cdot x$ equals $(g_1,g_2)F_0\cdot\Gamma.$   
 


\appendix
\section{Quasi-isometries of irreducible positive-rank arithmetic groups over function fields, by Kevin Wortman}

\subsection{Quasi-isometries}

Any finitely generated group has a left invariant word metric that is unique up to quasi-isometry. Recall that a quasi-isometry between metric spaces is a function $\phi : X
\rightarrow Y$ such that there are constants $L \geq
1$ and $C \geq 0$ for which any $x_{1},x_{2} \in X$:
$$\frac{1}{L} d\big(x_{1},x_{2}\big) -C \leq d\big(\phi(x_{1}),\phi(x_{2})\big)
\leq   L d\big(x_{1},x_{2}\big) +C$$
and such that every point in $Y$ is within distance $C$ of some point
in the image of $X$.

The surge of interest in viewing finitely-generated groups as geometric objects led to the ongoing investigation of which distinct groups are quasi-isometric to each other. Just as important is deciding all the self-quasi-isometries that a given group can exhibit, and for this it is convenient to define quasi-isometry groups:

 For a metric space $X$,
we define the relation $\sim$  on the set of functions  $X\rightarrow X$
by $\phi \sim \psi$ if
$$\sup _{x\in X} d\big(\phi(x),\psi(x)\big)<\infty$$
Then if $\Gamma$ is a finitely-generated group with a word metric, we
form the set of all quasi-isometries of $\Gamma$, and denote the
quotient modulo $\sim$ by $\mathcal{QI}(\Gamma)$. We call
$\mathcal{QI}(\Gamma)$ the \emph{quasi-isometry group} of $\Gamma$
as it has a natural group structure arising from function
composition.

\subsection{Quasi-isometries of lattices}

Ideally one would like to be able to determine the quasi-isometry group of any finitely-generated group. While in practice this is a difficult problem for a general group, there has been some very good progress made on this problem for certain special classes of  finitely-generated groups. Not surprisingly, since quasi-isometries originated in Mostow's study of strong rigidity, one of the special classes that was studied with success was the class of irreducible lattices in semisimple groups. Again, not surprisingly considering the origins of these problems in Mostow's work, it was the quasi-isometry groups of cocompact lattices that were approached and classified first.

The quasi-isometry groups of cocompact lattices in semisimple groups over nondiscrete locally compact fields was worked out independently by cases in the work of Mostow, Tukia, Koranyi-Reimann, Pansu, and Kleiner-Leeb \cite{Mo}, \cite{Tu}, \cite{KR}, \cite{P}, \cite{KL}. To summarize briefly, if $\Gamma$ is an irreducible cocompact lattice in a semisimple Lie group $G$, then $\mathcal{QI}(\Gamma)$ is isomorphic ``up to compact groups" to $G$ as long as the only rank one factors of $G$ are locally isomorphic to $\text{Sp}(n,1)$ and $F_4^{-20}$. If $G$ has any other rank one factor, then the quasi-isometry group does have a definable structure, but that structure varies by cases, and in any of those cases $\mathcal{QI}(\Gamma)$ is infinite-dimensional.

The classification of quasi-isometry groups of irreducible non-cocompact lattices in real semisimple Lie groups was worked out by Schwartz, Farb, and Eskin \cite{S}, \cite{F-S}, \cite{S2}, \cite{Es}, \cite{Fa}. Later, Taback and Wortman extended the classification to the setting of irreducible non-cocompact arithmetic lattices in Lie groups over nondiscrete, locally compact fields of characteristic zero \cite{Ta}, \cite{W2}, \cite{W3}. Roughly speaking, any quasi-isometry of a lattice described in this paragraph is a finite distance in the sup-norm from a commensurator, as long as the lattice in question is not commensurable with $\mathbf{PGL_2}(\mathbb{Z})$. To state the result precisely, we first need to organize some notation.

\subsection{Commensurators} 
We let $K$ be a global field, $V_K$ the set
of all inequivalent valuations on $K$, and $V_K^\infty \subseteq
V_K$ the subset of archimedean valuations. We will use $S$ to
denote a finite nonempty subset of $V_K$ that contains $V_K^\infty$, and we
write the corresponding ring of $S$-integers in $K$ as
$\mathcal{O}_S$.

For any valuation
$v\in V_K$, we let $K_v$ be the completion of $K$ with respect to
$v$. For any nonempty set of valuations $S \subseteq V_K$, and any algebraic $K$-group $\mathbf{G}$,
we define
$$G_{S} = \prod _{v \in S} \mathbf{G}(K_v)$$
We identify $\mathbf{G}(\mathcal{O}_S)$
 as a discrete subgroup of $G_S $ using the diagonal embedding.

We let $\text{Aut}(G_S)$ be the group of topological group
automorphisms of $G_S$. An automorphism $\psi \in \text{Aut}(G_S)$
\emph{commensurates} $\mathbf{G}(\mathcal{O}_S)$ if
$\psi(\mathbf{G}(\mathcal{O}_S)) \cap \mathbf{G}(\mathcal{O}_S)$
is a finite index subgroup of both $\psi
(\mathbf{G}(\mathcal{O}_S) )$ and $\mathbf{G}(\mathcal{O}_S)$.

We define the \emph{commensurator group of}
$\mathbf{G}(\mathcal{O}_S)$ to be the subgroup of
$\text{Aut}(G_S)$ consisting of automorphisms that commensurate
$\mathbf{G}(\mathcal{O}_S)$. This group is denoted as
$\text{Comm}_{\text{Aut}(G_S)}(\mathbf{G}(\mathcal{O}_S))$. Notice that it
differs from the standard definition of the commensurator group
of $\mathbf{G}(\mathcal{O}_S)$ in that we have not restricted
ourselves to inner automorphisms.

\subsection{Number fields}

The sum of the results mentioned above of Schwartz, Farb, Eskin, Taback, and Wortman are given by the following

\begin{thm}\label{t:a} Suppose $K$ is a global number field, and
$\mathbf{G}$ is a connected, absolutely simple, $K$-isotropic, algebraic $K$-group of adjoint type.
If either $K \ncong \mathbb{Q}$,
$S \neq V_K^\infty$, or $\mathbf{G}$ is not
$\mathbb{Q}$-isomorphic to $\mathbf{PGL_2}$, then there is an
isomorphism
$$ \mathcal{QI}(\mathbf{G}(\mathcal{O}_S) ) \cong \text{\emph{Comm}}_{\text{\emph{Aut}}(G_S)}(\mathbf{G}(\mathcal{O}_S))$$

\end{thm}

Note that Theorem~\ref{t:a} does not apply to $\mathbf{PGL_2}(\mathbb{Z})$ as this group is virtually free, and thus has an uncountable quasi-isometry group.

\subsection{Function fields} While Theorem~\ref{t:a} completely resolves the quasi-isometric classification of non-cocompact, irreducible, arithmetic groups over global number fields, most cases for the quasi-isometric classification of non-cocompact, irreducible, arithmetic groups over global function fields had remained open. The expected result is as follows:

\begin{conj}\label{c:b} Suppose $K$ is a global function field, and
$\mathbf{G}$ is a connected, absolutely simple, $K$-isotropic, algebraic $K$-group of adjoint type.
If $\sum _{v \in S} \text{\emph{rank}}_{K_v}(\mathbf{G})>1$, then
 there is an
isomorphism
$$ \mathcal{QI}(\mathbf{G}(\mathcal{O}_S) ) \cong \text{\emph{Comm}}_{\text{\emph{Aut}}(G_S)}(\mathbf{G}(\mathcal{O}_S))$$

\end{conj}

Note that the assumption from Conjecture~\ref{c:b} that $\sum_{v \in S} \text{rank}_{K_v}(\mathbf{G})>1$ is equivalent to the finite generation of the group $\mathbf{G}(\mathcal{O}_S)$, and thus is required in order for $\mathcal{QI}(\mathbf{G}(\mathcal{O}_S))$ to be defined.

A quick application of the paper that this note is an appendix to is to provide a proof for a significant portion of Conjecture~\ref{c:b} in the form of the following theorem.

\begin{thm}\label{t:c} Suppose $K$ is a global function field, and
$\mathbf{G}$ is a connected, absolutely simple, $K$-isotropic, algebraic $K$-group of adjoint type.
If $\text{\emph{rank}}_{K_v}(\mathbf{G})>1$ for each $v \in S$, then
 there is an
isomorphism
$$ \mathcal{QI}(\mathbf{G}(\mathcal{O}_S) ) \cong \text{\emph{Comm}}_{\text{\emph{Aut}}(G_S)}(\mathbf{G}(\mathcal{O}_S))$$

\end{thm}

\begin{proof} We denote by $\text{Aut}_{\text{Hd}}(G_S \,;\, \mathbf{G}(\mathcal{O}_S))$ the set of 
all $\psi \in \text{Aut}(G_S)$
such that the Hausdorff distance between $\mathbf{G}(\mathcal{O}_S)$ and
$\psi(\mathbf{G}(\mathcal{O}_S))  $ is finite. Proposition 6.9 of \cite{W2} states that 
$$\mathcal{QI}(\mathbf{G}(\mathcal{O}_S) ) \cong \text{Aut}_{\text{Hd}}(G_S \,;\, \mathbf{G}(\mathcal{O}_S)) $$ It is in the proof of this isomorphism from \cite{W2}  where the assumption that $\text{{rank}}_{K_v}(\mathbf{G})>1$ for each $v \in S$ is essential.

We want to show that $$\text{Aut}_{\text{Hd}}(G_S \,;\, \mathbf{G}(\mathcal{O}_S)) \cong  \text{Comm}_{\text{Aut}(G_S)}(\mathbf{G}(\mathcal{O}_S))$$ which proves the theorem.

Let $\Lambda =\mathbf{G}(\mathcal{O}_S)$. We may assume that $\Lambda $ is contained $G_S^+$, which is the subgroup of $G_S$ generated by the unipotent radicals of parabolic subgroups of the factor groups $\mathbf{G}(K_v)$. If we let $\text{Ad} : \mathbf{\widetilde{G}} \rightarrow \mathbf{G}$ be the universal covering of $\mathbf{G}$, then $\text{Ad} ^{-1} (\Lambda )$ is a lattice in $\widetilde{G}_S$. We denote  $\text{Ad} ^{-1} (\Lambda )$ by $\Gamma$.

Any $\psi \in \text{Aut}_{\text{Hd}}(G_S \,;\, \Lambda)$ corresponds to an automorphism $\widetilde{\psi} \in \text{Aut}(\widetilde{G}_S )$ that stabilizes $\Gamma$ up to finite Hausdorff distance. We let $\Gamma^*=
\widetilde{ \psi }( \Gamma)$.

By Theorem~\ref{orbitclosure}, ${\Delta( \widetilde{G}_S)(\Gamma, \Gamma ^*)}  $ is either closed or dense in $\widetilde{G}_S \times \widetilde{G}_S / \Gamma \times \Gamma ^*$

If it is dense, then for any $g \in \widetilde{G}_S$ there is a sequence $\{ g_k \}$ of elements in $\widetilde{G}_S$ such that $g_k \Gamma \to \Gamma$ and  $g_k \Gamma ^*\to g\Gamma^*$. Since $g_k \Gamma \to \Gamma$, there are $h_k \in  \widetilde{G}_S$ and $\gamma _k \in \Gamma$ such that $g_k=h_k \gamma _k$ and $h_k \to 1$. Therefore, $h_k^{-1}g_k \in \Gamma$ and $h_k^{-1}g_k \Gamma ^* \to g \Gamma *$. Notice that we have just proved that $\Gamma \Gamma ^*$ is dense in $\widetilde{G}_S / \Gamma^*$. But this is a contradiction. Indeed, the Huasdorff distance between $\Gamma$ and $\Gamma ^*$ in $\widetilde{G}_S$ is finite. Therefore, $\Gamma \Gamma ^*$ is a bounded subset of the unbounded space $\widetilde{G}_S / \Gamma^*$.

We are left to conclude that ${\Delta( \widetilde{G}_S)(\Gamma, \Gamma ^*)}  $ is closed, and we claim that $\Gamma \Gamma ^* \subseteq \widetilde{G}_S / \Gamma^*$ is closed. 

To prove our claim,  suppose we have a sequence $\{\gamma _k\} \subseteq \Gamma$ and a group element $g \in \widetilde{G}_S$ with $\gamma _k \Gamma ^* \to g \Gamma ^*$. Thus $\gamma _k(\Gamma, \Gamma ^*)\to (\Gamma ,  g \Gamma ^*)$ so that  $(\Gamma ,  g \Gamma ^*)= h (\Gamma ,  \Gamma ^*)$ for some $h \in \widetilde{G}_S $. Therefore, $h \in \Gamma$ and  $h \Gamma^* =g\Gamma^*$.

We have that, $\Gamma \Gamma ^*$ is a countable, closed, and bounded subset of $\widetilde{G}_S/\Gamma ^*$. Thus, $\Gamma \Gamma ^*$ is finite which is to say that $\Gamma$ and $\Gamma ^*$ are commensurable. It follows that $\psi$ commensurates $\Lambda$.

\end{proof}

Conjecture~\ref{c:b} has also been proved in the case when there are $v,w \in S$ such that $\text{rank}_{K_v}(\mathbf{G})=1$ and $\text{rank}_{K_w}(\mathbf{G})>1$. Indeed, the proof of the main theorem of \cite{W3} applies to this ``mixed rank" case essentially without modification. One simply needs to replace the role of the real semisimple Lie group factor in the proof from \cite{W3} with $\mathbf{G}(K_w)$ and replace the Borel reference for reduction theory of rank one arithmetic groups in real semisimple Lie groups with its counterpart for the positive characteristic case. (One account of this well-known counterpart can be found in \cite{BW}.) Otherwise, the proof needs no nonobvious modifications.

In light of Theorem~\ref{t:c} and the comment of the above paragraph, Conjecture~\ref{c:b} reduces to studying lattice actions on a product of trees:

\begin{conj}\label{c:d} Suppose $K$ is a global function field, and
$\mathbf{G}$ is a connected, absolutely simple, $K$-isotropic, algebraic $K$-group of adjoint type.
If $\text{\emph{rank}}_{K_v}(\mathbf{G})=1$ for all $v \in S$, and $|S|>1$, then
 there is an
isomorphism
$$ \mathcal{QI}(\mathbf{G}(\mathcal{O}_S) ) \cong \text{\emph{Comm}}_{\text{\emph{Aut}}(G_S)}(\mathbf{G}(\mathcal{O}_S))$$

\end{conj}

Note that \cite{W2} provides a step towards proving Conjecture~\ref{c:d} by proving that any quasi-isometry of $\mathbf{G}(\mathcal{O}_S)$ is induced by a factor-preserving quasi-isometry of the product of Bruhat-Tits trees corresponding to $G_S$.



\begin{thebibliography}{ZZZZ}

\bibitem[B91]{B1}A.~Borel,
{\it Linear algebraic groups.}
Second enlarged edition. Berlin Heidelberg, New York, Springer 1991.

\bibitem[BT65]{BT1}A.~Borel, J.~Tits,
{\it Groupes r\'{e}ductifs.}
Publications math\'{e}matiques de l' I.H.E.S, tome \textbf{27} (1965) 55-151 

\bibitem[BS68]{BS}A.~Borel, T.~A.~Springer,
{\it Rationality properties of linear algebraic groups.}
T$\hat{o}$hoku Math. J. \textbf{20} (1968) 443-497. 

\bibitem[BZ76]{BZ}I.~N.~Bernstein, A.~V.~Zelevinski,
{\it Representation of the group $\mbox{GL}(n, F)$ where $F$ is a non-archimedean local field.}
Russ. Math. Surv. \textbf{313} (1976) 1-68.

\bibitem[BW08]{BW}K.-U.~Bux, K.~Wortman, {\it Connectivity properties of horospheres in Euclidean buildings and applications to finiteness properties of discrete groups.} Preprint.

\bibitem[DM93]{DM}S.~G.~Dani, G.~A.~Margulis,
{\it Limit distributions of orbits of unipotent flows and values of quadratic forms.} 
I. M. Gelfand Seminar,  91--137, Adv. Soviet Math., \textbf{16}, Part 1, Amer. Math. Soc., Providence, RI, 1993.

\bibitem[EG]{Einsiedler-Ghosh} 
M.~Einsiedler and A.~Ghosh, 
{\it Rigidity of measures invariant under semisimple groups in positive characteristic.} Preprint (2008).

\bibitem[E98]{Es}A.~Eskin, {\it Quasi-isometric rigidity of nonuniform lattices in higher rank symmetric spaces.} J. Amer. Math. Soc., \textbf{11} (1998) 321-361.

\bibitem[EMS96]{EMS}A.~Eskin, S.~Mozes, N.~Shah,
{\it Unipotent flows and counting lattice points on homogeneous varieties}
Ann. of Math \textbf{143} (1996) 253-299.

\bibitem[F97]{Fa}B.~Farb, {\it The quasi-isometry classification of lattices in semisimple Lie groups.} Math. Res. Lett., \textbf{4} (1997) 705-717.

\bibitem[FSc96]{F-S}B.~Farb, R.~Schwartz, {\it The large-scale geometry of Hilbert modular groups.} J. Diff. Geom., \textbf{44} (1996) 435-478.

\bibitem[Gh05]{Gh}A.~Ghosh,
{\it Metric Diophantine approximation over a local field of positive characteristic},
 J. Number Theory  \textbf{124}  (2007),  no. 2, 454--469.


\bibitem[Gu81]{Gu} M.~de Guzman
{\it Real variable methods in Fourier analysis,}
Mathematics Studies, no. 46, Notas de Matem$\hat{a}$tica, North-Holland Publishing Company, 
Amsterdam, (1981)

\bibitem[KL97]{KL} B.~Kleiner, B.~Leeb, {\it Rigidity of quasi-isometries for symmetric spaces and Euclidean buildings}, Inst. Hautes \'{E}tudes Sci. Publ. Math. \textbf{86} (1997) 115-197.

\bibitem[KM98]{KM}D.~Y.~Kleinbock, G.~A.~Margulis,
{\it Flows on homogeneous spaces and Diophantine approximation on manifolds},
Annals of Math. \textbf{148} (1998) 339-360.

\bibitem[KR95]{KR}A.~Kor\'{a}nyi, H.~M.~Reimann, {\it Foundations for the theory of quasiconformal mappings on the Heisenberg group}, Adv. Math. \textbf{111} (1995) 1-87.

\bibitem[KT05]{KT}D.~Kleinbock, G.~Tomanov,
{\it Flows on $S$-arithmetic homogenous spaces and application to metric Diophantine approximation},  Comment. Math. Helv.  \textbf{82}  (2007),  no. 3, 519--581.

\bibitem [Mar71] {Mar1}G.~A.~Margulis,  
{\it On the action of unipotent groups in the space of lattices.} 
In Gelfand, I.M. (ed.) Proc. of the summer school on group representations. Bolyai 
Janos Math. Soc., Budapest, 1971, 365-370. Budapest: Akademiai Kiado (1975)

\bibitem[Mar86]{Mar2}G.~A.~Margulis,
{\it Indefinite quadratic forms and unipotent flows on homogeneous spaces.}
Proceed of ``Semester on dynamical systems and ergodic theory" (Warsa 1986) 399--409, Banach Center Publ., \textbf{23}, PWN, Warsaw, (1989).


\bibitem[Mar87b]{Mar3}G.~A.~Margulis, 
{\it Discrete subgroups and ergodic theory.} 
In Aubert, K.E. et al. (eds.) Proc. of the conference "Number theory, trace formula and discrete grops" 
in honor of A. Selberg. Oslo (1987), 377-388. London New York: Academic Press (1988) 

\bibitem[Mar90a]{Mar4}G.~A.~Margulis, 
{\it Orbits of group actions and values of quadrtic forms at integral points.} 
In Festschrift in honour of I.I. Piatetski-Shapiro. (Isr. Math. Conf. Proc.~vol. 3, pp. 127-151) Jerusalem: The Weizmann Science Press of Israel (1990) 

\bibitem[Mar90b]{Mar5}G.~A.~Margulis,  
{\it Discrete Subgroups of Semisimple Lie Groups.} 
Berlin Heidelberg New York: Springer (1990) 

\bibitem[Mar91]{Mar6}G.~A.~Margulis
{\it Dynamical and ergodic properties of subgroup actions on homogeneous spaces with applications to number theory.}  
Proceedings of the International Congress of Mathematicians, vol. I, II (Kyoto, 1990),  193--215, Math. Soc. Japan, Tokyo, 1991.

\bibitem[MT94]{MT}G.~A.~Margulis, G.~Tomanov,
{\it Invariant measures for actions of unipotent groups over local fields on homogeneous spaces. } Invent. Math.  \textbf{116}  (1994),  no. 1-3, 347--392.

\bibitem[M08]{AmirHoro} A. Mohammadi, 
{\it Mixing and its consequences in positive characteristic,}
Draft.

\bibitem[Mo68]{Mo}G.~D.~Mostow, {\it Quasi-conformal mappings in $n$-space and the rigidity of hyperbolic space forms.} Inst. Hautes \'{E}tudes Sci. Publ. Math. \textbf{34} (1968) 53-104.

\bibitem[Pa89]{P}P.~Pansu, {\it M\'{e}triques de Carnot-Carath\'{e}odory et quasiisom\'{e}tries des espaces sym\'{e}triques de rang un.}  Ann. of Math. \textbf{129} (1989)  1-60.

\bibitem[Pr77]{Pa}G.~Prasad, 
{\it Strong approximation for semi-simple groups over function fields,} 
Ann. of Math. (2), {\bf 105} (1977) 553-572 

\bibitem[R83]{Rat1}M.~Ratner, 
{\it Horocycle flows: joining and rigidily ol products.} 
Ann. Math. \textbf{118} (1983) 277-313
 
\bibitem[R90a]{Rat2} M.~Ratner, 
{\it Strict measure rigidity for unipotent subgroups of solvable groups.} 
Invent. Math. \textbf{101} (1990) 449-482 

\bibitem[R90b]{Rat3}M.~Ratner,  
{\it On measure rigidity of unipotent subgroups of semi-simple groups.} 
Acta, Math. \textbf{165} (1990) 229 -309  

\bibitem[R91]{Rat4}M.~Rather, 
{\it Raghunathan topological conjecture and distributions of unipotent flows.} 
Duke Math. J. \textbf{63} (1991) 235-280 

\bibitem[R92]{Rat5} M.~Ratner, 
{\it On Raghunathan's measure conjecture.} 
Ann. Math. \textbf{134} (1992) 545-607 

\bibitem[R95]{R3} M. Ratner, \textit{Raghunathan's conjectures for Cartesian products of real and $p$-adic Lie groups}, Duke Math. J. 77 (1995), no. 2, 275--382.


\bibitem[Sc95]{S}R.~Schwartz, \emph{The quasi-isometry classification of rank one lattices.} Inst. Hautes \'{E}tudes Sci. Publ. Math., \textbf{82} (1995)
133-168.

\bibitem[Sc96]{S2}R.~Schwartz, \emph{Quasi-isometric rigidity and Diophantine approximation.} Acta Math., \textbf{177} (1996)
75-112.

\bibitem[Sh94]{Sh}N.~Shah
{\it Limit distributions of polynomial trajectories on homogeneous spaces}
Duke. Math. J. \textbf{75} (1994) 711-732

\bibitem[Ta00]{Ta}J.~Taback, {\it Quasi-isometric rigidity for $PSL_{2}(\mathbb{Z}[1/p])$.} Duke Math. J., \textbf{101} (2000) 335-357.

\bibitem[To00]{T}G.~Tomanov,
{\it Orbits on homogeneous spaces of arithmetic origin and approximations.}  
Analysis on homogeneous spaces and representation theory of Lie groups, Okayama--Kyoto (1997),  265--297, Adv. Stud. Pure Math., \textbf{26}, Math. Soc. Japan, Tokyo, 2000.

\bibitem[Tu85]{Tu}P.~Tukia, {\it Quasiconformal extension of quasisymmetric mappings compatible with a M\"{o}bius group.} Acta Math. \textbf{154} (1985) 153-193.

\bibitem[W07]{W2}K.~Wortman, {\it Quasi-isometric rigidity of higher rank $S$-arithmetic lattices.} Geom. Topol. \textbf{11} (2007) 995-1048.

\bibitem[W08]{W3}K.~Wortman, {\it Quasi-isometries of rank one $S$-arithmetic lattices.} Preprint.


\end{thebibliography}
\end{document}